\documentclass{article}
\usepackage[utf8]{inputenc}
\usepackage{fullpage}
\usepackage{hyperref}
\usepackage{xcolor}
\usepackage{amsmath,amsfonts,amssymb}
\usepackage{graphicx}%
\usepackage{subcaption}
\usepackage{physics}
\usepackage{amsthm}
\usepackage[capitalize]{cleveref}
\usepackage{enumitem}
\usepackage{comment}
\newtheorem*{theorem-non}{Theorem}  
\usepackage{booktabs}

\newcommand{\R}{\mathbb{R}}

\newtheorem{theorem}{Theorem}
\newtheorem{proposition}[theorem]{Proposition}
\newtheorem{lemma}[theorem]{Lemma}

\newtheorem{definition}[theorem]{Definition}
\newtheorem{remark}{Remark}

\DeclareMathOperator{\F}{Tr}
\DeclareMathOperator{\diag}{diag}
\DeclareMathOperator{\loss}{Loss}
\DeclareMathOperator{\Loss}{loss}
\newcommand{\vecsigma}{\pmb{\sigma}}

\newcommand{\bfH}{{\bf H}}
\newcommand{\bfM}{{\bf M}}
\newcommand{\cond}{\mathrm{cond}}

\begin{document}
\title{Unique reconstruction for discretized inverse problems: a random sketching approach via subsampling}

\author{Ruhui Jin \thanks{Corresponding author, \textit{rjin@math.wisc.edu}} \thanks{Department of Mathematics, University of Wisconsin-Madison} , Qin Li$^\dagger$ , Anjali Nair\thanks{Department of Statistics and CCAM, University of Chicago} , and Samuel Stechmann$^\dagger$\thanks{Department of Atmospheric and Oceanic Sciences, University of Wisconsin-Madison}}
\date{ }

\maketitle
\begin{abstract}
Theoretical inverse problems are often studied in an ideal infinite-dimensional setting. The well-posedness theory provides a unique reconstruction of the parameter function, when an infinite amount of data is given. Through the lens of PDE-constrained optimization, this means one attains the zero-loss property of the mismatch function in this setting. This is no longer true in computations when we are limited to finite amount of measurements due to experimental or economical reasons. Consequently, one must compromise the goal, from inferring a function, to a discrete approximation. 

What is the reconstruction power of a fixed number of data observations? How many parameters can one reconstruct? Here we describe a probabilistic approach,
and spell out the interplay of the observation size $(r)$ and the number of parameters to be uniquely identified $(m)$. The technical pillar here is the random sketching strategy, in which the matrix concentration inequality and sampling theory are largely employed. By analyzing a randomly subsampled Hessian matrix, we attain a well-conditioned reconstruction problem with high probability. Our main theory is validated in numerical experiments, using an elliptic inverse problem as an example.

\end{abstract}


\section{Introduction}\label{sec:intro}
In many PDE-based inverse problems, measurements of a PDE model are used to infer unknown parameters (often a function) of a PDE. A very common numerical strategy is to use the framework of PDE-constrained optimization.

Under the constraint that the PDE model is honored, we are to minimize a loss function, defined as the mismatch between PDE-simulated data and the measured data:
\begin{equation}
\label{eqn:opt_cont}
\begin{array}{l}
\displaystyle \min_{\sigma\in L^2(\Omega)}\loss[\sigma]= \frac{1}{2} \, \left\|\mathcal{M}[\sigma] -d\right\|^2_{L^2(\mathcal{D})}\,,\quad\text{where}\quad d=\mathcal{M}[\sigma_*]\in L^2(\mathcal{D}),\\ \displaystyle \text{where}~~\mathcal{M}[\sigma] =  \omega[u_\sigma]=\omega \circ \mathcal{L}_\sigma^{-1} [S],
\end{array}
\end{equation}
where $\sigma(x)\in L^2(\Omega)$ is the parameter function to be inferred with $x$ supported in $\Omega$ (typically assumed to be compactly supported and with $C^1$ boundary). The mismatch is quantified by the functional $\loss: L^2(\Omega) \to \mathbb{R}^+$. The parameter-to-observation map is $\mathcal{M}: L^2(\Omega) \to L^2(\mathcal{D})$, where the measured location is denoted by the set $\mathcal{D}$ and is not necessarily the same as $\Omega$. This map $\mathcal{M}$ is read as follows. The PDE solution is $u_\sigma(x)$, as a function of $x$, for PDE operator $\mathcal{L}_\sigma$ that is parameterized by $\sigma(x)$ using the source term $S(x)$, and $\omega[u_\sigma]$ takes the measurement of the solution $u_\sigma$, and returns another function. Note that only a single source $S(x)$ is denoted in (\ref{eqn:opt_cont}) for notational simplicity, but the formulation can easily incorporate situations where multiple sources present. We are to compare this measured PDE-simulated solution with the true measurement data $d(x)$ in $L^2(\mathcal{D})$ norm, assuming the data has no noise pollution, i.e. $d = \mathcal{M}[\sigma_*]$. Since both the to-be-reconstructed parameter $\sigma: \Omega \to \R$, and the available data $d: \mathcal{D} \to \R$ are functions, the problem is posed in the infinite-to-infinite dimensional setting.

In computations, this formulation is adjusted to fit a discrete setting. In practice, let the basis of interest be a subset of a complete basis:
\[
\text{span} \left\{ \phi_1, \phi_2, \dots, \phi_N\right\}~~ \subseteq ~~ \text{span} \left\{ \phi_1, \phi_2, \dots\right\} = L^2(\Omega)\,.
\]
Then the to-be-reconstructed function $\sigma \in L^2(\Omega)$ is approximated by $\vecsigma = [\sigma_1, \dots, \sigma_N] \in \R^N,$ with $\sigma \approx \sum_{i=1}^N \sigma_i \phi_i $. The number of available data points is also finite. Letting $\{\psi_j\}_{j=1}^r$ be the collection of test functions used to detect the PDE solution, the discrete finite-to-finite dimensional problem is reformulated as:
\begin{equation}
\label{eqn:opt_dis}
\min_{\vecsigma\in \mathbb{R}^N}\Loss(\vecsigma)=\sum_{j=1}^r\underbrace{\frac{1}{2}\left( \langle\mathcal{M}(\vecsigma), \psi_j\rangle -{\bf d}_j\right)^2}_{\Loss_j} \,,\quad\text{where}\quad {\bf d}_j= \langle\mathcal{M}(\vecsigma_*), \psi_j\rangle_{\mathcal{D}}\in\mathbb{R}.
\end{equation}
Here we abuse the notation and set $\mathcal{M}(\vecsigma)=\mathcal{M}[\sum_{i=1}^N\sigma_i\phi_i]$. The collected data points $\{{\bf d}_j\}$ are assumed to be noise-free, and thus have the form of ${\bf d}_j=\langle\mathcal{M}(\vecsigma_*), \psi_j\rangle$ where $\vecsigma_*$ is the ground-truth media. In this case, the ground-truth $\vecsigma^*$ is a global optimum to \eqref{eqn:opt_dis}, since it achieves zero-loss: $\text{loss}(\vecsigma_*) = 0$ . The notation of ``$\Loss: \R^N \to \R^+$'' is no longer a functional, but a real-valued function.

\subsection{Motivation and objective of this work}
On the surface, the continuous~\eqref{eqn:opt_cont} and the discrete~\eqref{eqn:opt_dis} problems are rather similar. However, their theoretical behaviors are actually drastically different. In particular, close to the ground-truth $\sigma_*$, by deploying the unique reconstruction theory proved in the analytical setting, we can deduce that the Hessian of the cost functional is not degenerate, ensuring that $\sigma_*$ is the unique optimum point in the neighborhood. In the discrete case, however, eliminating this degeneracy requires further assumptions.

To be more specific, first consider the continuous setting. Denoting $\mathcal{H}\loss[\sigma_*]$ as the Hessian of the functional $\loss$ at $\sigma_*$, then formally for every pair of $(\widetilde{\sigma}_1,\widetilde{\sigma}_2)$, the Hessian is a bilinear map defined as
\begin{equation*}
\mathcal{H}\loss[\sigma_*](\widetilde{\sigma}_1,\widetilde{\sigma}_2) = \lim_{\delta\to0}\frac{\loss[\sigma_*+\delta\widetilde{\sigma}_1+\delta\widetilde{\sigma}_2]+\loss[\sigma_*]-\loss[\sigma_*+\delta\widetilde{\sigma}_1]-\loss[\sigma_*+\delta\widetilde{\sigma}_2]}{\delta^2}\,,
\end{equation*}
and thus:
\[
\mathcal{H}\loss[\sigma_*]:L^2(\Omega)\times L^2(\Omega)\to\mathbb{R}\,.
\]
With mild assumptions (\ref{hess positivity}), one deduces:
\begin{equation}
    \label{eqn:hessian_infinite}
\mathcal{H} \loss[\sigma_*]\succ 0\,,\quad\text{or equivalently, for any }\widetilde{\sigma}\neq 0\in L^2(\Omega)\,, \quad \mathcal{H} \loss[\sigma_*](\widetilde{\sigma},\widetilde{\sigma})> 0\,.
\end{equation}
As a consequence, $\sigma_*$ is a local optimum point.

This strict positivity cannot be straightforwardly extended to fit the discrete setting \eqref{eqn:opt_dis}. Indeed, the loss function in~\eqref{eqn:opt_dis} also induces a Hessian. According to the definition of $\Loss=\sum_{j=1}^r \Loss_j$, with the parameter $\vecsigma\in\mathbb{R}^N$, the Hessian is
\begin{equation}
\label{eqn:hessian_N}
\mathcal{H} \Loss (\vecsigma) = \sum_{j=1}^r\mathcal{H} \Loss_j(\vecsigma)\in\mathbb{R}^{N\times N}\,.
\end{equation}
Naturally the Hessian is of size $\mathbb{R}^{N\times N}$ and has a rank at most $r$. The positivity is not guaranteed. Especially in the case of $N>r$, when the number of parameters to be recovered is more than the number of data points given, the Hessian is guaranteed to be rank-deficient. The statement on the uniqueness of the local minimum of $\vecsigma_*$ is no longer valid.

This vast contrast between positivity of the Hessian in the continuous setting~\eqref{eqn:hessian_infinite} and that in the discrete setting~\eqref{eqn:hessian_N} prompts the key question of the current work:

\begin{center}
    \emph{Given fixed data, under what circumstance is the positivity of the Hessian~\eqref{eqn:hessian_N} preserved?}
\end{center}

We should note that our framework takes the perspective of PDE-constrained optimization, and examines the Hessian's local behavior around the global minimum. Preserving positivity of the Hessian is only one of the first quantities to be examined. As having a well-conditioned Hessian makes the PDE-based optimization~\eqref{eqn:opt_dis} easier to solve. A more substantial research effort is required to produce an accurate approximation of the true minimizer, and to complete the full picture of a well-posedness comparison of the inverse problem posed in continuous and discrete settings.

\subsection{Key strategy and main results}
\label{subsec: strategy results}
Throughout the paper, we assume $N>r$ and thus it is impossible to reconstruct the full parameter $\vecsigma \in \R^N$. However, it is possible to recover a sector of it. We are to understand how big this sector can be.

Suppose we reduce our problem from recovering $\vecsigma\in\mathbb{R}^N$ to a subset of $\vecsigma|_{\Lambda} \in \R^m$ with the cardinality of $|\Lambda| = m$, with $\Lambda$ standing for a mask of the smaller set of parameters under reconstruction. If $m\leq r$, heuristically, it is tempting to say that $\vecsigma |_{\Lambda}$ can be recovered. However, the data can be ``degenerate." The relation between the number of available data points ($r$), and the number of parameters needing to be reconstructed ($m$) is intriguing. In this paper, we present a probabilistic approach to partially answer this question. We examine the possibility of ``sketching" the $N \times N$ matrix \eqref{eqn:hessian_N} and projecting it down to a smaller matrix, and we look for conditions under which the projected matrix attains positivity, making the problem locally strictly convex when confined in $\Lambda$.

The main technical tool we deploy comes from randomized matrix solvers and matrix sketching. In the current context, it is reformulated as subsampling the Hessian matrix: 
\begin{equation}
\label{eqn:hessian_reduce}
\mathcal{H} \Loss(\vecsigma)\in\mathbb{R}^{N\times N} \to \mathcal{H} \Loss(\vecsigma\vert_{\Lambda}) =\mathbf{S}(\mathcal{H}\Loss (\vecsigma)) \mathbf{S}^\top \in \R^{m \times m}.
\end{equation}
Here the ``subsampling'' matrix is $\mathbf{S} \in \R^{m \times N}$, which selects rows in $\mathcal{H}\Loss$ (see explicit definition in~\eqref{eqn: sub phi}).

We are to investigate the conditions on $\mathbf{S}$ so that the sketched submatrix~\eqref{eqn:hessian_reduce} has a high probability of being positive definite. To ensure the sketched Hessian is of full rank, and noting that the base Hessian \eqref{eqn:hessian_N} is rank $r$, it is necessary to assume $m\leq r$. Then, by labeling the eigenvalues w.r.t. the Euclidean norm of \eqref{eqn:hessian_reduce} in descending order as
\[
\lambda_1\geq\lambda_2\geq\cdots\geq\lambda_{m}\,,
\]
the condition number of the matrix can be written in conventional notation as
\begin{equation}
\label{eqn: cond}
\cond\left(\mathbf{S}(\mathcal{H}\Loss (\vecsigma)) \mathbf{S}^\top\right)=\frac{\lambda_1}{\lambda_m}\,.
\end{equation}
An informal version of \cref{thm: well-conditioned hess} is as follows: 
\begin{theorem-non}[exposition of main \cref{thm: well-conditioned hess}]
Consider the Hessian matrix $\mathcal{H}\Loss(\vecsigma) \in \R^{N \times N}$ \eqref{eqn:hessian_N} of rank $r$. Fix the sampling dimension $m \in \mathbb{N}^+$, where $m \leq o(r) \leq N$.
The subsampled Hessian ${\bf S}\,\mathcal{H}\Loss(\vecsigma)\, {\bf S}^\top \in \R^{m \times m}$ \eqref{eqn:hessian_reduce} is well-conditioned with high probability. Particularly, 
\begin{equation}
\label{eqn: success}
\mathbb{P} \left( \emph{\text{cond}} \left(\mathbf{S}(\mathcal{H}\Loss (\vecsigma)) \mathbf{S}^\top \right) \leq  o(r)\right) \leq 1 -\frac{1}{r}.
\end{equation}
The constants in the $o$ notations are to be explicitly spelled out in the main Theorem~\ref{thm: well-conditioned hess}.
\end{theorem-non}
Broadly speaking, this theorem demonstrates that the condition number of the Hessian submatrix is bounded above with high probability. We give a few noteworthy interpretations of this theorem: 
\begin{enumerate}
\item This theorem states that, to achieve the well-conditioning of the subsampled Hessian with high probability, the sampling dimension satisfies $m=o(r)$. The dependence is linear. In most situations, $r$ represents the number of measurements, so the presence of more measurements gives higher probability of securing a well-conditioned sketched problem.
\item The specific criterion of ``success," according to the LHS of \eqref{eqn: success}, 
is defined as the condition number of the sketched Hessian being bounded by a certain threshold. This threshold linearly grows with respect to the full rank $r.$
    \item The RHS of the concentration inequality \eqref{eqn: success} provides an explicit success probability. The greater the value of the full rank $r$ is, the greater the probability to achieve a well-conditioned subsampled Hessian \eqref{eqn:hessian_reduce}.
    
\end{enumerate}

\subsection{Literature review}

The problem is closely related to PDE-constrained optimization, and its technical support mostly comes from random matrix theory. To some extent, the problem can also be cast in a setting of experimental design, sometimes referred to as data selection. We now review relevant literature from all three perspectives. These three areas are heavily researched independently. We only provide references that are closest to our immediate study.

{\bf\emph{PDE-constrained optimization.}}
The problem that we study is formulated and originated from PDE-constrained optimization \cite{HPUU08}, a formulation extensively used for inverse problems \cite{OHS13,A99} and optimal control \cite{LEGHRSUU12}. There are many challenges. Most problems suffer from notorious non-convexity, with landscapes sometimes very rugged. Often, regularization is introduced to convexify the loss function~\cite{EHN96,IJ14,Heinz_Engl_2000,Engl_1989}, and in other approaches, specifics of the problems are built in algorithm design, as is done in full-wave inversion using multi-frequency data~\cite{EAGB08,CMOBV15,zhang_wideband} that mitigate cycle skipping~\cite{EY22}.

{\bf \emph{Randomized sketching and subsampling techniques.}}
Random matrix sketching provides the technical foundation for our current work. 
Randomized numerical linear algebra has attracted a lot of attention in the past decade, mainly in the context of data analysis and machine learning tasks. With connections across the fields of numerical linear algebra \cite{HMT11,M11,MT20}, applied probability \cite{V18}, and theoretical computer science \cite{FKV98,W14}, randomized linear algebra techniques and theory have largely been developed and have ripened in recent years. Broadly speaking, the subject is to look for a low-dimensional representation of high-dimensional data. This is supported by the bulk of theoretical foundations built upon the celebrated Johnson-Lindenstrauss lemma \cite{JL84}. The tools have also been investigated in the inverse setting, with the aim being to compute the standard linear regression $\mathbf{A}\mathbf{x}=\mathbf{b}$ with subsampled data from $\bf A, b$. This can be understood as a reduced inverse problem since one is essentially looking for an approximation to $\mathbf{A}^{-1}$ with the reduced rows.

In the context of the PDE-constrained optimizations, randomized solvers also have already facilitated the downstream iterative steps, for instance in waveform inversions \cite{S08,BHT11,CVG17}. Randomized matrix solvers thus have been successfully applied in preconditioning the Hessian; see works in \cite{CD12,DLBCCS12,FTU23}. Our project targets at a different focus from the randomized Hessian preconditioner. To be specific, we utilize the sketching probabilistic analysis to study the conditioning of the reduced Hessian matrix \eqref{eqn:hessian_reduce}. The specific machinery we lean on is the matrix concentration inequalities~\cite{RV07,T08-1,tropp2012user} for the spectral properties of the row and column sampled matrices. 

{\bf \emph{Experimental design and data selection.}}
Experiment design \cite{P06,TanLiZepeda_bridging} is a popular topic over the last few decades, as experimentalists ask for the most informative locations for placing detectors/sensors~\cite{brunton_sensor,Barthorpe_sensor}. The optimal design problem emerges from many different areas such as in biology, medical imaging \cite{bsewb23}, geophysics \cite{A21}, and many others. Our problem shares similar goals with experimental design, in the sense that both look to exploit information in the available data through the Hessian behavior~\cite{huan2024optimalexperimentaldesignformulations,jin2024optimaldesignlinearmodels}. In contrast to experimental design that looks for the most informative data for reconstructing a fixed set of parameters, our approach assumes that the data is already given, and we seek parameters that are reconstructable using the given data.

\subsection*{Organization of the paper}
The remainder of the paper is organized as follows. In Section~\ref{sec: main results} we set notations and present our main theoretical results. In particular, we quantify the probability of success, defined as the conditioning of the confined Hessian matrix being upper bounded, and we also evaluate the dependence of such success on various parameters (notably $m$ and $r$). In Section~\ref{sec:technical} we garner major technical components of the proof for the theorem. Finally in Section~\ref{sec:numerics} we present numerical examples over synthetic data and PDE simulation data from an elliptic inverse problem. The numerical results showcase the agreement with the theoretical guarantees. 
\section{Main results}
\label{sec: main results}
We present the formulation of the problem and unify notation in Section~\ref{sec: formulation}. The formulation leads to a problem of analyzing the conditioning of a specifically defined matrix. The main results are presented in subsection~\ref{sec:conditioning}. We discuss the proof ingredients and the general strategy in subsection~\ref{subsec: proof ingredients}. 

\subsection{Problem formulation}
\label{sec: formulation}
We formulate the problem and unify notation in this section. Recalling the definition in~\eqref{eqn:opt_dis}, with straightforward derivation, we compute each entry of the Hessian of the loss function. At the ground-truth $\vecsigma=\vecsigma_*$, for $i_1, i_2 = 1, \dots, N$, we have
\begin{equation}
\label{eqn: hess derivation}
\begin{array}{ll}
&\displaystyle (\mathcal{H}\,\text{loss})_{i_1, i_2} (\pmb{\sigma}_*) \\
\displaystyle&= \sum_{j=1}^r \langle \partial_{i_1} \mathcal{M}(\pmb{\sigma}_*),  \psi_j \rangle\,\langle \partial_{i_2} \mathcal{M}(\pmb{\sigma}_*),  \psi_j \rangle + (\mathcal{H} \mathcal{M})_{i_1, i_2}(\pmb{\sigma}_*)  (\langle\mathcal{M}(\pmb{\sigma}_*), \psi_j\rangle-{\bf d}_j),\\
\displaystyle&= \sum_{j=1}^r \langle \partial_{i_1} \mathcal{M}(\pmb{\sigma}_*),  \psi_j \rangle\,\langle \partial_{i_2} \mathcal{M}(\pmb{\sigma}_*),  \psi_j \rangle,\end{array}
\end{equation}
where the second term drops out because ${\bf d}_j=\langle\mathcal{M}(\vecsigma_*), \psi_j\rangle$, recovering the form of the Gauss-Newton Hessian. The first-order derivative has the standard definition $\partial_i \mathcal{M}(\pmb{\sigma}) = \lim_{\epsilon \to 0}\frac{\mathcal{M}[\sigma+\epsilon \phi_i] - \mathcal{M}[\sigma] }{\epsilon}$ and when $\{\phi_i\}$ are orthonormal, $ \partial_i \mathcal{M}(\pmb{\sigma})= \langle\frac{\delta\mathcal{M}}{\delta\sigma}\,,\phi_i\rangle_\Omega$ where $\frac{\delta\mathcal{M}}{\delta\sigma}$ is the Fr\'echet derivative.

This symmetric structure of the Hessian~\eqref{eqn: hess derivation} strongly suggests the use of the short-hand notation of
\begin{equation}
\label{eqn: hess N}
\begin{array}{l}
{\bf H}^N  \displaystyle = \left[\begin{array}{ccc}
{\bf H}_{1,1} &\dots &{\bf H}_{1,N}\\
 \vdots & \ddots& \vdots \\
{\bf H}_{N,1} & \dots  &{\bf H}_{N,N}
\end{array}\right]  = \displaystyle \pmb{\Phi}^N \pmb{\Phi}^{N\top}
\in \R^{N \times N},
\\\text{with}\quad \pmb{\Phi}^N \in \R^{N \times r}~~\text{and}~~ \pmb{\Phi}_{i,j} = \langle \partial_i \mathcal{M}(\pmb{\sigma}_*),  \psi_j \rangle.
\end{array}
\end{equation}

We are interested in the regime where $N\gg r$, so many rows in $\pmb{\Phi}^N \in \R^{N \times r}$ are linearly dependent, and $\bf H$ is rank deficient and non-invertible.

We look for a subsampled Hessian matrix that can be inverted, and has moderate condition number. Since the Hessian matrix is always positive semidefinite, we are to look for the subsampling that guarantees \emph{strict} positive definiteness. Mathematically, randomly selecting $m\leq r$ rows from $\pmb{\Phi}^N$ amounts to multiplying the sampling matrix ${\bf S}^N \in \R^{m \times N}$ on $\pmb\Phi^N$:
\begin{equation}
\label{eqn: sub phi}
{\bf S}^N \pmb{\Phi}^N = 
\left[\begin{array}{c}
- \pmb{\Phi}^N_{\omega_1,:} - \\
\vdots \\
- \pmb{\Phi}^N_{\omega_m,:} -
\end{array}\right] \in \R^{m \times r},\quad\text{with}\quad{\bf S}^N_{i,:} = e_{\omega_i}\,,
\end{equation}
where $e_i$ is the unit vector placing $1$ on $i$-th entry, and the random variables $\{\omega_i\}_{i=1}^m$ are generated uniformly from the integer set $\{1, \dots, N \}$.

Accordingly, the sampled Hessian submatrix becomes:
\begin{equation}
\label{eqn: sub hess}
{\bf S}^N {\bf H}^N {\bf S}^{N\top} = {\bf S}^N \pmb{\Phi}^N \pmb{\Phi}^{N\top} {\bf S}^{N\top}:={\bf H}^{N}_s \in \R^{m \times m}.
\end{equation}
The problem now translates to finding conditions for this subsampled matrix ${\bf H}^{N}_s $ to be well-conditioned.

\subsection{Conditioning of the Hessian}
\label{sec:conditioning}
We are to estimate the upper bound of the condition number of ${\bf H}^N_s$ when $m$ is small. To start, we first define a few key quantities for properties of ${\bf H}^N$. 
\begin{definition}
\label{def: parameters}
\begin{enumerate}
\item{Diagonal variation parameters:}
We denote by $\ell$ and $L$ the diagonal variation parameters. In particular, supposing all the diagonals are strictly positive, we define
\begin{equation}
\label{eqn: diag bound}
\ell: = \frac{N}{\F\left({\bf H}^N\right)} \min_{i =1,\dots, N} {\bf H}_{i,i}^N, \quad \quad 
L:=\frac{N}{\F\left({\bf H}^N\right)} \max_{i =1,\dots, N} {\bf H}_{i,i}^N\,.
\end{equation}
\item{The coherence parameter:} We denote by $ \mu>0$ the coherence:
\begin{equation}
\label{eqn: coherence}
\mu :=\frac{N}{\left\|{\bf H}^N \right\|_F}\max_{i_1 \neq i_2 = 1, \dots, N} \left\vert{\bf H}^N_{i_1, i_2} \right\vert\,.
\end{equation}
\end{enumerate}
\end{definition}
We note that all these quantities are inherent attributes of ${\bf H}^N$ and are ``relative'' in nature. To be specific, the quantities are preserved when $\bfH^N$ is scaled by a constant. For example, $L(\bfH^N)=L(c\bfH^N)$ for any $c >0$. Clearly, $\frac{\F\left({\bf H}^N\right)}{N}$ is the averaged value along the diagonal for $\bfH^N$, and $\frac{\left\|{\bf H}^N \right\|_F}{N}$ is the root-mean-square of the entries in $\bfH^N$.

We next provide the main probabilistic statement of the Hessian submatrix ${\bf H}^N_s$ being well-conditioned. This is the rigorous version of the informal statement shown in \cref{subsec: strategy results}.
\begin{theorem}
\label{thm: well-conditioned hess}
Consider the discrete Hessian matrix ${\bf H}^N \in \R^{N \times N}$ of rank $r$ from \eqref{eqn: hess N}. Fix the sampling dimension $m \in \mathbb{N}^+,$ where $m\leq r \leq N$, and moreover let $m$ satisfy 
\begin{equation}
\label{eqn: m est}
m \leq \displaystyle \min\left\{\frac{\ell}{146e^{\frac{1}{4}}}\,\left(\frac{\F\left({\bf H}^N\right)}{\left\|{\bf H}^N\right\|_2}-1\right)\,,\quad \frac{\ell^2}{149e^{\frac{1}{2}}\mu^2 \log r} \, \frac{\F\left({\bf H}^N\right)^2}{\left\|{\bf H}^N \right\|^2_F}\right\}\,.
\end{equation}
The subsampled Hessian ${\bf H}^N_s \in \R^{m \times m}$ \eqref{eqn: sub phi}-\eqref{eqn: sub hess}, generated from uniform sampling, is well-conditioned with high probability: 
\begin{equation}
\label{eqn: main probability}
\mathbb{P}\left( \cond(\bfH^N_s) \leq \frac{L+\tau(m)}{\ell - \tau(m)} \right) \geq 1-\frac{1}{r}\,,
\end{equation}
where the condition number is defined in \eqref{eqn: cond}, and the distortion quantity $\tau(m)$ is
\begin{equation}
\label{eqn: distortion}
\tau(m) = e^{\frac{1}{4}} \, \left(2m\, \frac{\left\|{\bf H}^N\right\|_2}{\F\left({\bf H}^N\right)} + 12\mu\, \sqrt{m\, \log r}\, \frac{\left\|{\bf H}^N\right\|_F}{\F\left({\bf H}^N\right)}\right).
\end{equation}
\end{theorem}
We leave the proof to the later part of this section. Besides the general interpretations previously stated in (i)-(iii) \cref{subsec: strategy results}, we provide more thorough explanations below.
\begin{itemize}
\item 
The condition number threshold appearing in \eqref{eqn: main probability} is approximately $\frac{L}{\ell}$, the ratio of the maximal and the minimal diagonal entries in ${\bf H}^N$, with perturbations encoded in $\tau$ that we call the distortion term. Notably, $\tau(m)$ is asymptotically linear in the sample size $m$ for large $m$. Intuitively, a smaller sample size $m$ brings less distortion $\tau(m)$ and then the conditioning is closer to $\frac{L}{\ell}$. This matches the intuition that the system behaves better when the number of to-be-reconstructed parameter is small.
\item The success probability is $1-\frac{1}{r}$. This means bigger $r$ gives higher probability for the subsampled Hessian to be well-conditioned. This also matches the intuition that the system behaves better when one has a large number of measurements.
\item To make sense of the theorem, we also need the term $\ell - \tau(m)$ to be strictly greater than $0$. This assumption is satisfied by the requirement on $m$ in \eqref{eqn: m est}. A more careful derivation also suggests
\begin{equation}
\label{eqn: condition threshold}
0< \frac{L+\tau(m)}{\ell-\tau(m)}\leq \frac{73r(L+\ell)}{\ell}.
\end{equation}
Please see \cref{proof: moment bound} for the complete proof of the claim above.
\item Noticing that ${\bf H}^N$ has rank $r$, due to Cauchy-Schwarz, we obtain the basic estimates of
\[
1\leq \frac{\F\left( {\bf H}^N\right)}{\left\|{\bf H}^N \right\|_2} = \frac{\sum_{i=1}^r \lambda_i}{\max_{i} \lambda_i} \leq r,\quad  1\leq \frac{\F\left( {\bf H}^N\right)^2}{\left\|{\bf H}^N \right\|_F^2} = \frac{(\sum_{i=1}^r \lambda_i)^2}{\sum_{i=1}^r \lambda_i^2} \leq r.\]
This suggests the upper bound for $m$ in \eqref{eqn: m est} is roughly $\mathcal{O}(r)$ up to a constant factor. When the matrix is incoherent, meaning with a large $\mu$, such constant scaling becomes worse. \cref{thm: well-conditioned hess} concludes that when the number $m$ (of to-be-reconstructed parameters) scales linearly with the number of measurements $r$, then with high probability, the global minimizer is locally unique and the problem is well-conditioned in its neighborhood.
\end{itemize}

\subsection{Proof ingredients}
\label{subsec: proof ingredients}
In this subsection, we illustrate the key technical results and prepare for the proof of \cref{thm: well-conditioned hess}. 

We use the following short-hand notation for the off-diagonal Gram matrix:
\[
{\bf M}^N = {\bf H}^N - \text{diag}\left({\bf H}^N\right) \in \R^{N \times N}. 
\]
Similarly, the sketched submatrix of ${\bf M}^N$ is denoted as
\begin{equation}
\label{eqn: M_s}
{\bf M}_s^N  = {\bf H}_s^N -\text{diag} \left({\bf H}_s^N\right) = {\bf S}^N {\bf H}^N {\bf S}^{N\top} - \text{diag}\left({\bf S}^N {\bf H}^N {\bf S}^{N\top}\right) \in \R^{m \times m}.  
\end{equation}

\begin{lemma}
\label{prob comparison}
Under the same assumption as in \cref{thm: well-conditioned hess}, the probability of the well-conditioned Hessian has the following upper bound:
\begin{equation}
\label{eqn: prob comparison}
\mathbb{P}\left( \cond\left({\bf H}^N_s\right) \leq \frac{L+\tau(m)}{\ell- \tau(m)} \right)
\geq \mathbb{P}\left(\left\|{\bf M}^N_s \right\|_2 \leq \frac{\F\left({\bf H}^N\right)}{N}\,\tau(m) \right). 
\end{equation}
\end{lemma}
The proof of \cref{prob comparison} is in \cref{proof: prob comparison}. This lemma transfers our focus from analyzing the conditioning of $\bfH^N_s$ to controlling the tail bound of $\|{\bf M}^N_s \|_2$. As we shall see below, this spectral norm quantity $\|{\bf M}^N_s \|_2$ is indeed a sub-Gaussian random variable, hence a sub-Gaussian type concentration inequality can be applied. We now recall the specific definition.

\begin{proposition}[sub-Gaussian distribution (Proposition 10 of \cite{T08-1})]
\label{def: subgaussian}
Let $\alpha, \beta, D>0$ be finite numbers such that the non-negative random variable $X$ satisfies
\begin{equation}
\label{eqn: subgaussian condition}
\left(\mathbb{E} X^p\right)^{\frac{1}{p}} \leq \alpha \sqrt{p} + \beta, \quad \text{for~all~} p \geq D.
\end{equation}
Then for all $u \geq \sqrt{D},$ one has the tail probability
\begin{equation}
\label{eqn: subgaussian tail}
\mathbb{P} \left\{X \geq  e^{\frac{1}{4}}\,\left(\alpha u + \beta \right)\right\} \leq e^{-\frac{u^2}{4}}.
\end{equation}
\end{proposition}
There are a number of equivalent definitions for sub-Gaussian and the characterization of the tail bound. We refer the readers to Proposition 2.5.2 in \cite{V18}.

In the following lemma, we provide the moment estimation of the random variable $\| {\bf M}^N_s\|_2$ and show that it exhibits a sub-Gaussian decay by its higher moment bound.
\begin{lemma}
\label{moment bound}
Under the same assumption as in \cref{thm: well-conditioned hess}, for $p\geq 2,$ the $p$-th moment is bounded by 
\begin{equation}
\label{eqn: moment bound}
\begin{array}{ll}
\displaystyle \left(\mathbb{E} \left\| {\bf M}^N_s\right\|_2^p\right)^{\frac{1}{p}} 
& \displaystyle \leq \frac{2m}{N} \left\|{\bf H}^N\right\|_2  + 12\sqrt{\max(\log m, \frac{p}{2})} \,\mu\, \frac{\sqrt{m}}{N}\,\left\| {\bf H}^N\right\|_F.
\end{array}
\end{equation}
\end{lemma}
The proof of~\Cref{moment bound} is in \cref{proof: moment bound}.

Equipped with these two lemmas, we proceed with the proof for \cref{thm: well-conditioned hess}.
\begin{proof}[Proof of \cref{thm: well-conditioned hess}]
Based on the moment estimation \eqref{eqn: moment bound} in \cref{moment bound}, the random variable $\| {\bf S}^N {\bf M}^N {\bf S}^{N\top}\|_2$ is of sub-Gaussian type. In particular, according to \cref{def: subgaussian}, we set 
\[
\beta = \frac{2m}{N} \left\|{\bf H}^N\right\|_2, \quad \alpha = 6\mu\frac{\sqrt{m}}{N}\,\left\| {\bf H}^N\right\| _F, \quad D = 4\log m.
\]
Therefore, for $u = 2\sqrt{\log r} > \sqrt{D} =2\sqrt{\log m}, (r \geq m)$ by the distortion term $\tau$ \eqref{eqn: distortion}, we achieve the probability tail bound 
\begin{equation}
\label{eqn: technical prob tail}
\begin{array}{l}
\displaystyle\mathbb{P}\left( \left\|{\bf M}^N_s\right\|_2 \leq\frac{\F\left({\bf H}^N\right)}{N}\,\tau(m)\right) \\
= \displaystyle\mathbb{P}\left( \left\|{\bf M}^N_s\right\|_2 \leq 
e^{\frac{1}{4}}\left(\frac{m}{N}\left\|{\bf H}^N \right\|_2 + 6\mu \frac{\sqrt{m}}{N}\left\| {\bf H}^N\right\|_F\,2\sqrt{\log r}\right)\right)\\
\geq \displaystyle 1- e^{-\frac{4 \log r}{4}} =1- \frac{1}{r}.
\end{array}
\end{equation}
Applying the events relation in \cref{prob comparison}, we prove the general well-conditioning Hessian probability  \eqref{eqn: main probability},
\[
\mathbb{P}\left( \cond\left({\bf H}^N_s \right) \leq \frac{L+\tau(m)}{\ell -\tau(m)} \right) \geq \mathbb{P}\left( \left\|{\bf M}^N_s\right\|_2 \leq\frac{\F\left({\bf H}^N\right)}{N}\,\tau(m)\right)\geq 1- e^{-\frac{4 \log r}{4}} =1- \frac{1}{r}.\]

This completes the proof. 
\end{proof}

We make a final remark on potential improvement for the main \cref{thm: well-conditioned hess}.
\begin{remark}
\label{remark: use of parameters}
In our proof, the way that the parameters $L, \ell$ \eqref{eqn: diag bound} are employed is pessimistic. Indeed, their definitions are to look at the extreme large and the extreme small ratio compared to the average diagonal entries, so they represent the worst case scenario. If we are to examine the distribution of diagonal entries, most entries are centered around the average. Hence the worst case bound is not tight, i.e., the extreme condition imposed is more stringent than necessary for most other entries. It is clear that the entry magnitude is well-concentrated although there is a big disparity between $\ell$ and $L$.

To obtain a better bound, a quick relaxation to this requirement is to examine a high-probability cut-off: Fix a small failure probability $0<\eta< \frac{1}{2}$, and define the modified version of $\ell_0, L_0>0$ as
\begin{equation}
\label{eqn: failure prob}
\mathbb{P}\left\{ \min_{i =1,\dots, m} \left({\bf H}^N_s\right)_{i,i} <  \frac{\ell_0}{\frac{\F({\bf H}^N)}{N}}  \right\} \leq \frac{\eta}{2}, \quad \mathbb{P}\left\{ \max_{i =1,\dots, m}  \left({\bf H}^N_s \right)_{i,i} > \frac{L_0}{\frac{\F({\bf H}^N)}{N}}  \right\} \leq \frac{\eta}{2}.
\end{equation}
Clearly, we have $\ell_0 \geq \ell, \quad L_0 \leq L$. With these re-defined quantities, the success probability \eqref{eqn: main probability} in the main theorem is improved to:
\[
\mathbb{P}\left( \cond\left({\bf H}^N_s \right) \leq \frac{L_0+\tau(m)}{\ell_0 - \tau(m)} \right) \geq 1-\frac{1}{r} - \frac{\eta}{2} - \frac{\eta}{2} =1-\frac{1}{r} - \eta.
\]
where the condition number threshold has a more moderate upper bound. In~\cref{fig: diagonal bound} we show $\ell_0$ and $L_0$ with cut-off probability $\frac{\eta}{2}=0.1$. In this case, the bound $\frac{L_0}{\ell_0}= \frac{143.9}{63.2} \approx 2.28$ is an improvement over $\frac{L}{\ell}=\frac{163}{57.7} \approx 2.83$, and the sacrifice on the probability is $\eta =0.2$.

\begin{figure}[!htb]
\centering
\includegraphics[width=12cm]{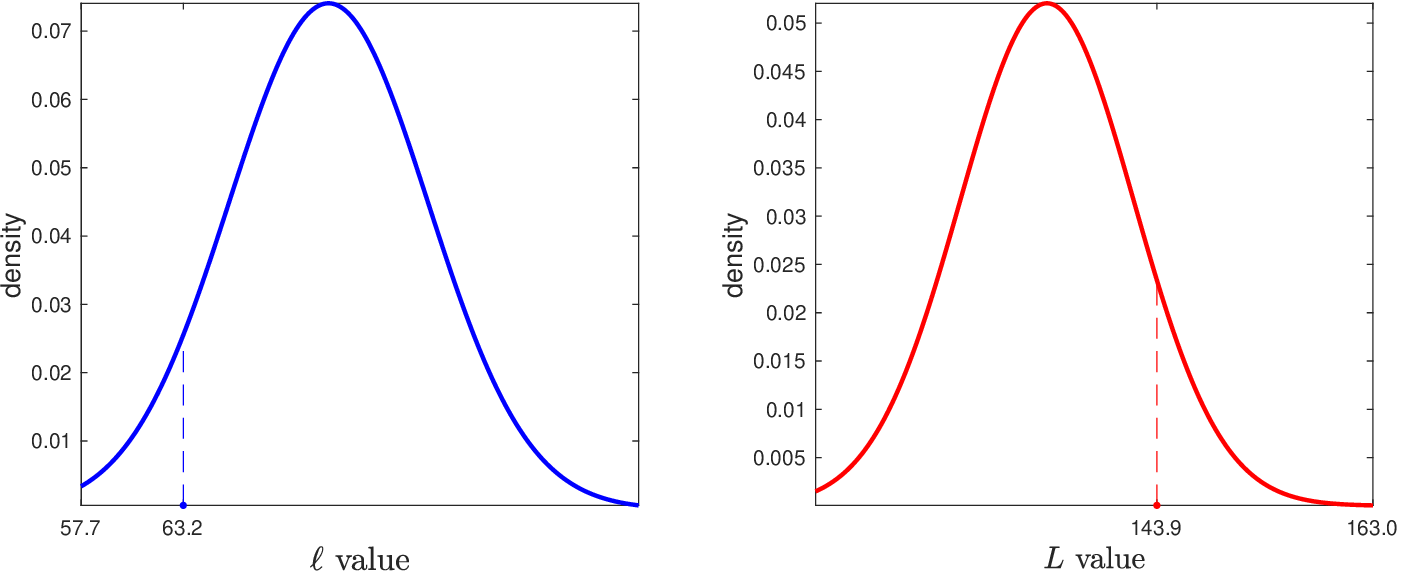}
\caption{Distributions of $\min$ and $\max$ diagonal entries in $\bfH^N_s$ out of 10,000 simulations. We sample $m = 30$ from the Gauss-Newton Hessian ${\bf H}^N = \pmb{\Phi}^N \pmb{\Phi}^{N\top}$ where $\pmb{\Phi}^N\in \R^{N \times r}$ is generated by a random Gaussian matrix in the size of $N = 5000, r = 100.$} 
\label{fig: diagonal bound}
\end{figure}
\end{remark}
\section{Details of the proofs}\label{sec:technical}
The main theorem is built upon two technical lemmas:~\cref{prob comparison} on estimating the probability of well-conditioning, and~\cref{moment bound} on estimating high moments of an intermediate variable $\|\bfM_s^N\|_2$~\eqref{eqn: M_s}. We present their proof in the following two subsections respectively.

\subsection{Probability bound, \cref{prob comparison}}
\label{proof: prob comparison}
In order to certify a good conditioning of the subsampled Hessian ${\bf H}^N_s$, \cref{prob comparison} informs us that the key quantity is the spectral norm of the hollow Gram matrix ${\bf M}^N_s$ in \eqref{eqn: M_s}. We establish this connection and prove \cref{prob comparison} as follows.

\begin{proof}[Proof of \cref{prob comparison}]
Recall  
\[
{\bf M}^N_s= {\bf H}^N_s- \text{diag}\left({\bf H}^N_s\right) \in \R^{m \times m}.
\]
By Weyl's inequality (originally in \cite{W12}, or Corollary 6.3.8 of \cite{HJ91}), for $k = 1, \dots, m,$
\begin{equation}
\label{eqn: sigma perturb}
\left\vert\lambda_k\left({\bf H}^N_s \right) - \lambda_k\left(\text{diag}\left({\bf H}^N_s\right)\right) \right\vert\leq \left\|{\bf M}^N_s\right\|_2. \end{equation}
Take $k = 1,$ and using the reverse triangle inequality, we know that
\[
\begin{array}{ll}
\displaystyle \lambda_1\left({\bf H}^N_s \right) & \displaystyle\leq\lambda_1\left(\text{diag}\left({\bf H}^N_s \right)\right) +\left\|{\bf M}^N_s \right\|_2\\
&\displaystyle =  \max_{i = 1, \dots, m}\left({\bf H}^N_s\right)_{i,i} +\left\|{\bf M}^N_s \right\|_2\\
& \displaystyle \leq  L\frac{\F\left({\bf H}^N \right)}{N}+\left\|{\bf M}^N_s\right\|_2.\end{array}
\]
The last inequality is obtained due to \eqref{eqn: diag bound}. Analogously as above derivations, let $k=m$, we have 
\[
\lambda_m\left({\bf H}^N_s\right)\geq \ell\frac{\F\left({\bf H}^N \right)}{N}-  \left\|{\bf M}^N_s\right\|_2.
\]
The event 
\begin{equation}
\label{eqn: event SMS^T}
\left\{\left\|{\bf M}^N_s \right\|_2 \leq \frac{\F\left({\bf H}^N \right)}{N}\,\tau(m) \right\} 
\end{equation}
implies the outcomes:
\[
\left\{
\begin{array}{l}
\displaystyle \lambda_1\left( {\bf H}^N_s \right)\leq L\frac{\F\left({\bf H}^N \right)}{N} +\left\| {\bf M}^N_s \right\|_2 \leq  \frac{\F\left({\bf H}^N \right)}{N}\left(L+\tau(m)\right),\\
\displaystyle \lambda_m\left( {\bf H}^N_s \right)\geq  \ell \frac{\F\left({\bf H}^N \right)}{N} -\left\|{\bf M}^N_s \right\|_2 \geq\frac{\F\left({\bf H}^N \right)}{N}\left(\ell-\tau(m)\right) .\\
\end{array}\right.
\]
And consequently, assuming $\ell-\tau(m)>0$, the
event \eqref{eqn: event SMS^T} leads to another event: 
\[
\left\{\cond\left({\bf H}^N_s \right) = \frac{\lambda_1}{\lambda_m} \left({\bf H}^N_s \right)\leq \frac{L+\tau(m)}{\ell- \tau(m)}\right\}. 
\]
Thus we have the probability relation of the two events:
\[
\mathbb{P}\left( \cond\left({\bf H}^N_s\right) \leq \frac{L+\tau(m)}{\ell- \tau(m)} \right)
\geq \mathbb{P}\left(\left\| {\bf M}^N_s \right\|_2 \leq \frac{\F\left({\bf H}^N\right)}{N}\,\tau(m) \right)\,.
\]
To complete the proof, we still need to show that~\eqref{eqn: m est} leads to $\ell-\tau(m)>0$. To do so, we recall from~\eqref{eqn: distortion}:
\begin{equation}
\tau(m) =\underbrace{2e^{\frac{1}{4}}m\frac{\left\|{\bf H}^N\right\|_2}{\F\left({\bf H}^N\right)}}_{\text{Term I}} + \underbrace{12\mu e^{\frac{1}{4}} \sqrt{m\, \log r}\frac{\left\|{\bf H}^N \right\|_F}{\F\left({\bf H}^N\right)}}_{\text{Term II}}\,.
\end{equation}
According to the condition on the sampling dimension $m$ \eqref{eqn: m est}, 
\[
\text{Term I}\leq  \frac{\ell}{73} \left( 1 - \frac{\|{\bfH}^N \|_2}{\F({\bfH}^N)}\right), \quad \text{Term II} \leq \frac{12 \ell}{\sqrt{149}} \leq \frac{72 \ell}{73}.
\]
Consequently:
\[
\tau(m) \leq \ell - \frac{\ell\|{\bfH}^N \|_2}{73 \F({\bfH}^N)}, ~~\text{and}~~\ell-\tau(m)\geq\frac{\ell\|{\bfH}^N \|_2}{73\F({\bfH}^N)}>0\,.
\]
This completes the proof. 
Moreover, a more careful derivation derives: 
\[
\frac{L + \tau(m)}{\ell - \tau(m)} \leq  \frac{L+\ell }{\frac{\ell\|{\bfH}^N \|_2}{73\F({\bfH}^N)} } =\frac{73 (L+\ell)}{\ell}\frac{\|{\bfH}^N \|_2}{\F({\bfH}^N)} \leq \frac{73 r\,(L+\ell)}{\ell},
\]
proving~\eqref{eqn: condition threshold}. 

\end{proof}

\subsection{Moment bound \cref{moment bound}}
\label{proof: moment bound}
\cref{moment bound} is to quantify the probability of $\left\|{\bf M}^N_s \right\|_2$ being controlled within a certain range. To get this concentration tail, it is helpful to understand the moments of $\left\|{\bf M}^N_s \right\|_2.$ Since higher moment bounds are crucial for characterizing this random variable’s distribution, we next pin down its concentration property. 

We remark that \cref{moment bound} is a generalization of Eqn. (6.1) of Tropp \cite{T08-1}. The key difference is that we remove the assumption in \cite{T08-1} that $\pmb{\Phi}^N$ has equal-norm rows. Nevertheless, we present the essential analysis for the moment bound \eqref{eqn: moment bound}. The proof framework uses the decoupling technique below. 
\begin{lemma}[Decoupling - Theorem 9 of \cite{T08-1}]
Fix an even integer $N \in \mathbb{N}.$ There exists a submatrix ${\widehat{\bf M}}^N \in \R^{\frac{N}{2} \times \frac{N}{2}}$ of ${\bf M}^N \in \R^{N \times N}$ so that
\begin{equation}
\label{eqn: decoupling}
\left(\mathbb{E} \left\| {\bf S}^N{\bf M}^N  {\bf S}^{N\top}\right\|_2^p\right)^{\frac{1}{p}} \leq 2\,\max_{m_1+m_2=m,\atop m_1, m_2 \in \mathbb{N}^+} \left(\mathbb{E} \left\| {\bf S}_1^N {\widehat{\bf M}}^N  {\bf S}_2^{N\top}\right\|^p\right)^{\frac{1}{p}}.\end{equation}
Here, ${\bf S}_1^N \in \R^{m_1 \times \frac{N}{2}}$ and ${\bf S}_2^N \in \R^{m_2 \times \frac{N}{2}}$ are independent random sampling matrices with sample size respectively $m_1, m_2 \in \mathbb{N}^+.$
The maximization in RHS of \eqref{eqn: decoupling} is taken over all possible combinations of $m_1, m_2$ satisfying $m_1+m_2 = m.$
\end{lemma}

To quantify the moment of interest $\left\| {\bf S}^N {\bf H}^N {\bf S}^{N\top} \right\|_2$, we instead focus on analyzing 
\begin{equation}
\label{eqn: S1 S2}
\left(\mathbb{E} \| {\bf S}^N_1 {\widehat{\bf M}}^N  {\bf S}_2^{N\top}\|_2^p  \right)^{\frac{1}{p}} = \left( \mathbb{E}_1 \left[ \mathbb{E}_2\left[ \|{\bf S}^N_1 {\widehat{\bf M}}^N  {\bf S}_2^{N\top} \|_2^p~\Big \vert ~{\bf S}_1\right]\right]\right)^{\frac{1}{p}}. 
\end{equation}
Regarding that, we present the useful information of both column and row samplings. The result relates to one specific type of matrix norms defined below. 

\begin{definition}
For a matrix $\bf A$ with $N$ columns, we define the induced $1,2$-norm, or equivalently the largest column Euclidean norm: 
\[
\|{\bf A} \|_{1,2}:= \sup_{\|{\bf x} \|_2 =1, \atop {\bf x} \in \R^N} \| {\bf A x}\|_1 = \max_{j=1, \dots, N} \|{\bf A}_{:,j} \|_2.
\]
\end{definition}

Next we present the key technical results for uniform sampling. 

\begin{lemma}[Theorem 8 of \cite{T08-1}]
\label{sampling result}
For a matrix ${\bf A}$ with $N$ columns. Draw a random coordinate selector ${\bf S} \in \R^{m \times N}.$
\begin{equation}
\label{eqn: column sampling}
\left(\mathbb{E} \left\| {\bf A S}^\top \right\|^p_2\right)^{\frac{1}{p}}  \leq \sqrt{\frac{m}{N}}\, \|{\bf A} \|_2 + 3\, \sqrt{\max (2\log m, \frac{p}{2}) } \, \left\|{\bf A}\right\|_{1,2}.
\end{equation}
Similarly, for row sampling, given a matrix ${\bf B}$ with $N$ rows. Draw a random coordinate selector ${\bf S} \in \R^{m \times N}.$
\begin{equation}
\label{eqn: row sampling}
\left(\mathbb{E} \left\| {\bf S B} \right\|^p_2\right)^{\frac{1}{p}}\leq \sqrt{\frac{m}{N}} \|{\bf B} \|_2 + 3\, \sqrt{\max (2\log m, \frac{p}{2}) } \,\left\|{\bf B}^\top\right\|_{1,2}.
\end{equation}
\end{lemma}

\cref{sampling result} (from~\cite{T08-1}) estimates the spectral norm of a submatrix whose columns are in a uniformly random collection of column set in the original matrix. This lemma generalizes the previous results in \cite{R99} and \cite{RV07} to  moments of all orders regimes. 
The row sampling can be proved by simply taking the transpose of $\bf B$ and applying the column sampling result previously stated.

\begin{proof}[Proof of \cref{moment bound}]
The proof adopts the techniques from Section 6 in \cite{T08-1}. 

\begin{enumerate}
    \item 
    Guided by the quantity of interest \eqref{eqn: S1 S2}, we first condition on ${\bf S}_1$ and consider the random column sampling of ${\bf S}_2.$ Applying the spectral norm estimation of column sampling \eqref{eqn: column sampling}, 
\[
 \left(\mathbb{E}_2\left[ \left\|{\bf S}^N_1 {\widehat{\bf M}}^N {\bf S}_2^{N\top} \right\|_2^p~\Big \vert ~{\bf S}_1\right]\right)^{\frac{1}{p}} \leq 
 \sqrt{\frac{2m_2}{N}} \left\|{\bf S}_1^N {\widehat{\bf M}}^N \right\|_2 + 3\, \sqrt{\max (2\log m_2, \frac{p}{2})}\,\left\|{\bf S}^N_1 {\widehat{\bf M}}^N {\bf S}^{N\top}_2\right\|_{1,2}.
\]

It remains to study:
\begin{equation}
\label{eqn: *condition S1}
\begin{array}{ll}
\displaystyle\left(\mathbb{E} \| {\bf S}^N_1 {\widehat{\bf M}}^N  {\bf S}_2^{N\top}\|_2^p  \right)^{\frac{1}{p}}& \displaystyle = \left( \mathbb{E}_1 \left[ \mathbb{E}_2\left[ \|{\bf S}^N_1 {\widehat{\bf M}}^N  {\bf S}_2^{N\top} \|_2^p~\Big \vert ~{\bf S}_1\right]\right]\right)^{\frac{1}{p}} \\
& \displaystyle \leq \left(\mathbb{E}_1 \left(\sqrt{\frac{m_2}{\frac{N}{2}}} \left\|{\bf S}_1^N {\widehat{\bf M}}^N  \right\|_2 + 3\, \sqrt{\max (2\log m_2, \frac{p}{2})}\, \left\|{\bf S}^N_1 {\widehat{\bf M}}^N  \right\|^p_{1,2}\right)^p\right)^{\frac{1}{p}}\\
& \displaystyle \leq\sqrt{\frac{m_2}{\frac{N}{2}}}\underbrace{\left(\mathbb{E}_1  \left\|{\bf S}_1^N {\widehat{\bf M}}^N  \right\|_2^p \right)^{\frac{1}{p}}}_{\text{term~(I)}} + 3\, \sqrt{\max (2\log m_2, \frac{p}{2})}\, \underbrace{\left(\mathbb{E}_1  \left\|{\bf S}^N_1 {\widehat{\bf M}}^N  \right\|_{1,2}\right)^{\frac{1}{p}}}_{\text{term~(II)}}\end{array}
\end{equation}
The last line is by triangle inequality. 

\item
Next, we respectively tackle term (I) and (II) estimations. 
\begin{enumerate}
    \item[Term (I):]
    Applying the spectral norm estimation of row sampling \eqref{eqn: row sampling}, 
    \begin{equation}
    \label{eqn: *S1 I}
    \begin{array}{ll}\displaystyle \left(\mathbb{E}_1 \left\|{\bf S}^N_1 {\widehat{\bf M}}^N \right\|_2^p\right)^{\frac{1}{p}} & \displaystyle \leq \sqrt{\frac{m_1}{\frac{N}{2}}} \left\| {\widehat{\bf M}}^N \right\|_2 + 3\, \sqrt{\max \left( 2\log m_1, \frac{p}{2}\right)}\,\left\| {\widehat{\bf M}}^{N \top}\right\|_{1,2}.\end{array}
    \end{equation}
    
Given the definition of $\mu$ \eqref{eqn: coherence}, the column norm of ${\widehat{\bf M}}^N \in \R^{\frac{N}{2} \times \frac{N}{2}}$ has a uniform bound, that is
\[
\left\| {\widehat{\bf M}}^{N\top}\right\|_{1,2} \leq\sqrt{\frac{N}{2}}\, \mu\, \frac{\left\| {\bf H}^N\right\|_F}{N} = \mu\,\frac{\left\|{\bf H}^N\right\|_F}{\sqrt{2N}}.
\]
Then continue from \eqref{eqn: *S1 I}, the estimation of term (I) in \eqref{eqn: *condition S1} becomes 
\begin{equation}
\label{eqn: *term I est}
\left(\mathbb{E}_1 \left\|{\bf S}^N_1 {\widehat{\bf M}}^N \right\|_2^p\right)^{\frac{1}{p}} \leq \sqrt{\frac{m_1}{\frac{N}{2}}} \left\| {\widehat{\bf M}}^N \right\|_2 + 3\, \sqrt{\max \left( 2\log m_1, \frac{p}{2}\right)}\,\mu\, \frac{\left\| {\bf H}^N\right\|_F}{\sqrt{2N}}.
\end{equation}

\item[Term (II):]
For any random realization of ${\bf S}_1^N$, we have a uniform bound
\[
\left\|{\bf S}_1^N {\widehat{\bf M}}^N \right\|_{1,2}\leq\sqrt{m_1}\mu\,\frac{\left\| {\bf H}^N \right\|_F}{N} =\mu\,\frac{\sqrt{m_1}}{N}\left\| {\bf H}^N \right\|_F.
\]
Consequently, the $p$-th moment is within this uniform bound as well, i.e.
\begin{equation}
\label{eqn: *term II est}
\left(\mathbb{E}_1 \left\|{\bf S}_1^N {\widehat{\bf M}}^N \right\|^p_{1,2}\right)^{\frac{1}{p}} \leq \left(\left(\mu\,\frac{\sqrt{m_1}}{N}\left\| {\bf H}^N \right\|_F\right)^p\right)^{\frac{1}{p}} =\mu\,\frac{\sqrt{m_1}}{N}\left\| {\bf H}^N \right\|_F.
\end{equation}
\end{enumerate}

\end{enumerate}

Now we plug in the analysis of term (I) \eqref{eqn: *term I est} and (II) \eqref{eqn: *term II est} back into the overall estimation \eqref{eqn: *condition S1},
\begin{equation}
\label{eqn: main est line 1}
\begin{array}{ll}
\displaystyle\left(\mathbb{E} \| {\bf S}^N_1 {\widehat{\bf M}}^N  {\bf S}_2^{N\top}\|_2^p  \right)^{\frac{1}{p}} & \displaystyle\leq \frac{\sqrt{m_1m_2}}{\frac{N}{2}}\, \left\| {\widehat{\bf M}}^N \right\|_2 + 3\, \sqrt{\max \left( 2\log m_1,2\log m_2, \frac{p}{2}\right)}\,\mu\,\frac{\sqrt{m_1}+\sqrt{m_2}}{N}\,\left\| {\bf H}^N \right\|_F\\
& \displaystyle\leq\frac{m}{N}\, \left\| {\bf M}^N \right\|_2 + 6\, \sqrt{\max \left( \log m, \frac{p}{2}\right)}\,\mu\,\frac{\sqrt{m}}{N}\,\left\| {\bf H}^N \right\|_F.
\end{array}
\end{equation}
To derive the last inequality, note that $\| {\widehat{\bf M}}^N\|_2 \leq \|{\bf M}^N \|_2,$ since ${\widehat{\bf M}}^N$ is the submatrix of ${\bf M}^N$. 
We further replace $m_1, m_2$ in the last inequality using the relation that $m_1+m_2 =m, m_1, m_2 \in \mathbb{N}^+$ which implies
\[
\max(\log m_1, \log m_2) < \log m, \quad 2\sqrt{m_1m_2} \leq m, \quad \sqrt{m_1} + \sqrt{m_2} \leq \sqrt{2m}.
\]

Moreover, due to the relation of ${\bf M}^N$ and ${\bf H}^N:$ 
\[
 {\bf H}^N - {\bf M}^N = \diag\left({\bf H}^N\right), \]
 and the fact that the difference $\diag\left({\bf H}^N\right)$ is positive semi-definite, by Weyl's monotonicity principle \cite{W12, B01}, 
 \[
\left\|{\bf M}^N\right\|_2 =\sigma_1 \left({\bf M}^N \right) \leq \sigma_1 \left({\bf H}^N\right)= \left\|{\bf H}^N\right\|_2. 
\]
We replace the norms of ${\bf M}^N$ in \eqref{eqn: main est line 1} with norm of ${\bf H}^N$ and obtain the main result \eqref{eqn: moment bound}:
\[
\left(\mathbb{E} \| {\bf S}^N_1 {\widehat{\bf M}}^N{\bf S}_2^{N\top}\|_2^p  \right)^{\frac{1}{p}} \leq \frac{m}{N}\, \left\| {\bf H}^N\right\|_2 + 6\, \sqrt{\max \left( \log m, \frac{p}{2}\right)}\,\mu\,\frac{\sqrt{m}}{N}\,\left\| {\bf H}^N\right\|_F.
\]
To obtain the ultimate moment bound, we trace back to the decoupling \eqref{eqn: decoupling}. Multiplying the above estimation with $2,$ we conclude the calculation \eqref{eqn: moment bound}. 
\[
\left(\mathbb{E} \left\| {\bf S}^N {\bf M}^N {\bf S}^{N\top}\right\|_2^p\right)^{\frac{1}{p}} \leq  \frac{2m}{N}\, \left\| {\bf H}^N\right\|_2 + 12\, \sqrt{\max \left( \log m, \frac{p}{2}\right)}\,\mu\,\frac{\sqrt{m}}{N}\,\left\| {\bf H}^N\right\|_F.\]
\end{proof}
\section{Numerical experiments}
\label{sec:numerics}
In this section, we present numerical evidence that corroborate our theory. According to \cref{thm: well-conditioned hess}, we have stated that, from \eqref{eqn: main probability}:
\begin{equation*}
\mathbb{P}\left(\cond\left({\bf H}^N_s\right) \leq \frac{L + \tau(m)}{\ell- \tau(m)} \right) \geq 1 - \frac{1}{r}.
\end{equation*}
The theory implies that with high probability, the condition number of the sketched Hessian is well within control. The success probability is better if $r$ is big (large number of measurements), and the condition number upper bound is smaller if $m$ is small (small number of to-be-reconstructed parameters). We validate these observations in our numerical simulation, and focus our attention on the interplay between the data rank $r$ and the sampling size $m$.

The following two subsections present numerical results run on two test cases: synthetic tests, and PDE-based simulation for an elliptic inverse problem. To obtain probabilistic results, $10,000$ independent runs are applied in each test.

\subsection{Synthetic tests}

We start with the synthetic experiment.
The full Hessian matrix ${\bf H}^N = \pmb{\Phi}^N \pmb{\Phi}^{N\top} \in \R^{N \times N}$ \eqref{eqn: hess N} is built via the rank-$r$ matrix $\pmb{\Phi}^N \in \R^{N \times r}$. The entries of $\pmb{\Phi}^N$ are generated by synthetic data, from various of choices of probability distributions shown below. We study the conditioning of the Hessian submatrix ${\bf H}^N_s = {\bf S}^N \pmb{\Phi}^N\pmb{\Phi}^{N\top}{\bf S}^{N\top} \in \R^{m \times m}$ \eqref{eqn: sub hess} after the uniform sampling applied on ${\bf S}^N\pmb{\Phi}^N$ \eqref{eqn: sub phi}. 

\noindent\textbf{Fixed $r$ with varying $m$.} In the first set of experiments, we work with a fixed value of the full rank $r$ (data points), and vary the sampling dimension $m.$ This testing scenario allows us to fix the probability bound on the right hand side of~\eqref{eqn: main probability} and check the cut-off value of condition number. The synthetic data \eqref{eqn: hess N} is built in three types. Each entry of the factor matrix $\pmb{\Phi}^N \in \R^{N \times r}$ is drawn from, respectively, standard Gaussian $\mathcal{N}(0,1)$, uniform $[0,1]$ and Bernoulli $\{0, 1\}$ distributions. $10,000$ independent subsamplings have been run for various values of $m$, and in \cref{fig: synthetic fix r}, we plot the percentage range $[20\%, 80\%]$ and the median value of the condition number of the sketched Hessian sub-matrices, for all three types of synthetic data matrix. The two subplots show results for $r=50$ and $r=100$ respectively. 

\begin{figure}[!htp]
    \centering
\subfloat[$r = 50$]{\includegraphics[width=8.5cm]{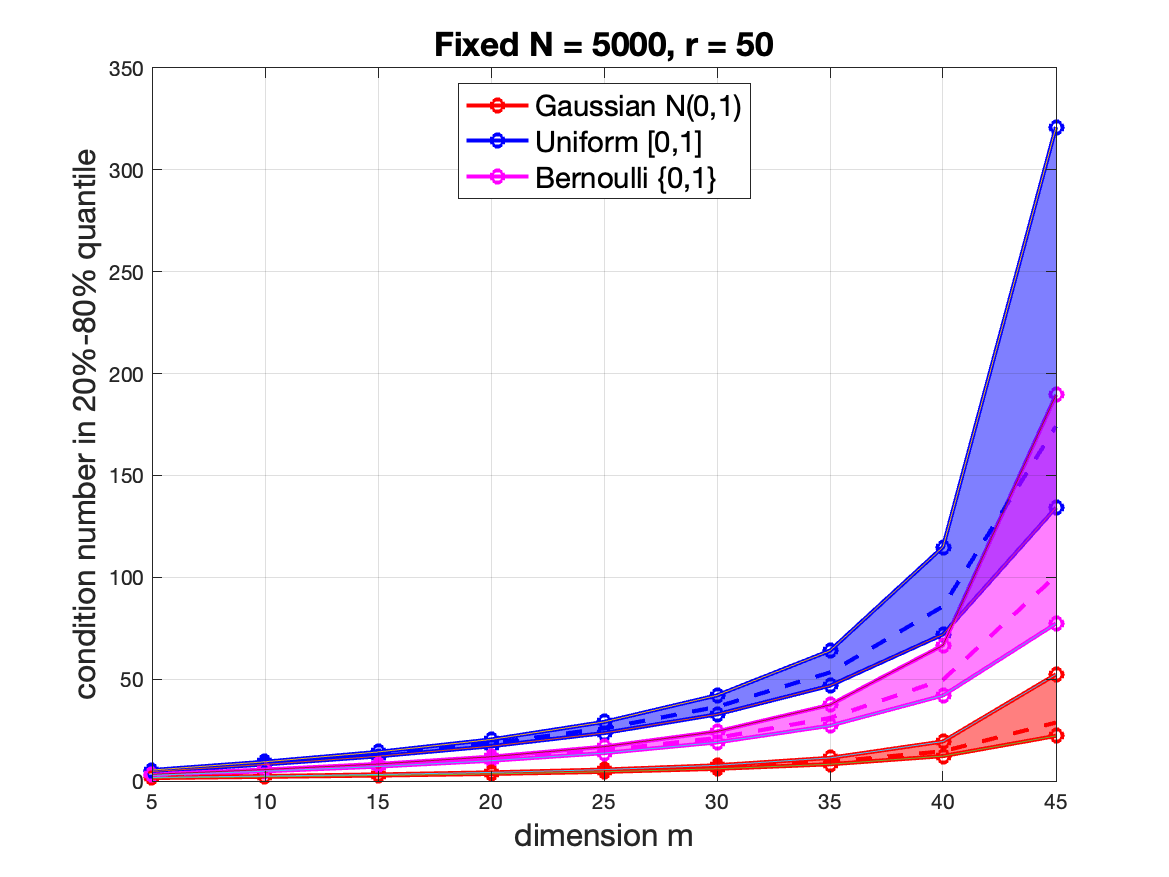}}~~
\subfloat[$r = 100$]{\includegraphics[width=8.5cm]{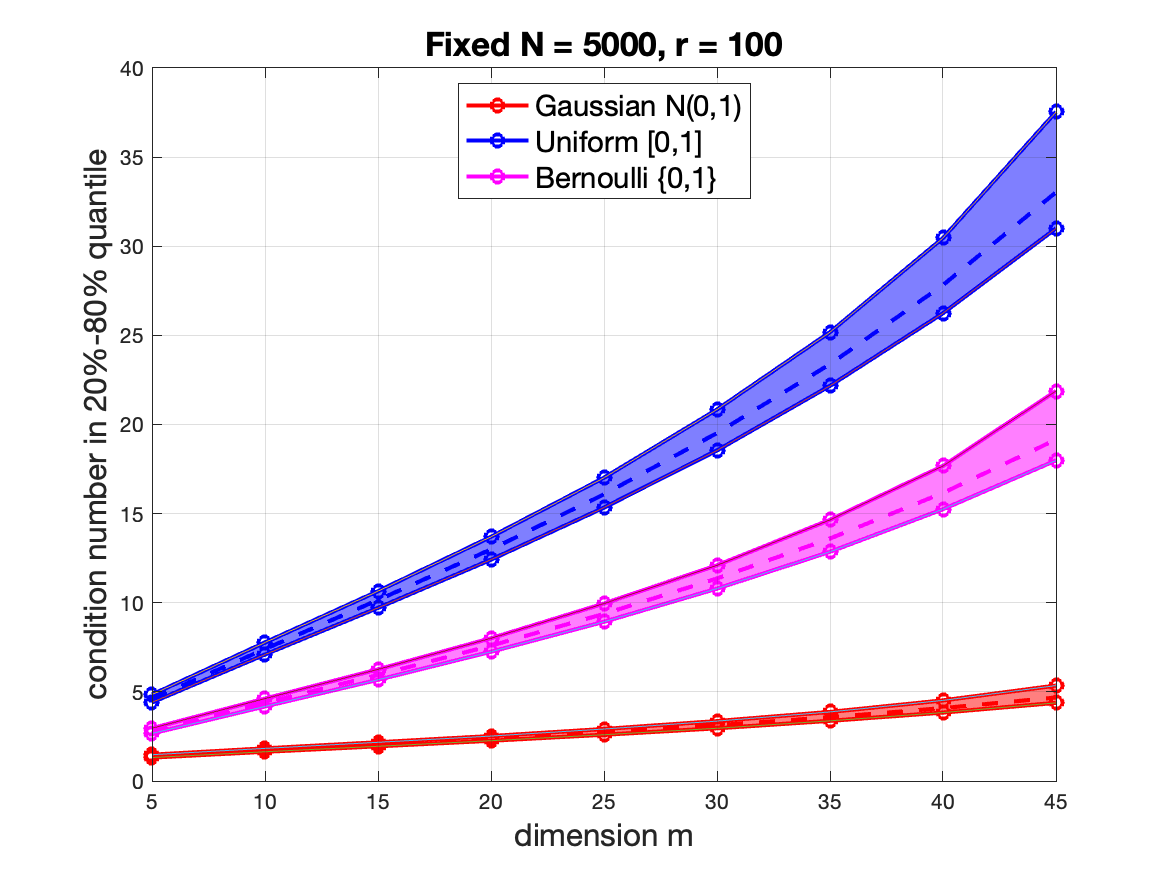}}
\caption{The concentration of condition number $\cond ({\bf H}^N_s).$}
\label{fig: synthetic fix r}
\end{figure}

{\begin{center}
\footnotesize
\begin{tabular}{cccc} 
\toprule
\vspace{1mm}
 & Gaussian  & Uniform  & Bernoulli\\
 \vspace{1mm}
 $\frac{L}{\ell}$ &4.1076&2.4819&3.0833\\
$\frac{\|{\bf H}^N \|_2}{\Tr({\bf H}^N)}$ &\color{blue} \bf 0.0241& \color{red} \bf 0.7547 & 0.5095\\
$\frac{\|{\bf H}^N \|_F}{\Tr({\bf H}^N)}$ &\color{blue} \bf 0.1422 &\color{red} \bf 0.7556 & 0.5143
\\
\toprule
\end{tabular}
~~~~~~
\begin{tabular}{cccc} 
\toprule
\vspace{1mm}
 & Gaussian  & Uniform  & Bernoulli\\
 \vspace{1mm}
 $\frac{L}{\ell}$ & 2.8263 & 2.0104 & 2.3000\\
 $\frac{\|{\bf H}^N \|_2}{\Tr({\bf H}^N)}$ &\color{blue} \bf 0.0130&\color{red} \bf 0.7521& 0.5043\\
$\frac{\|{\bf H}^N \|_F}{\Tr({\bf H}^N)}$ &\color{blue} \bf 0.1010&\color{red} \bf 0.7525&0.5068
\\
\toprule
\end{tabular}
\captionof{table}{Key parameter values in \eqref{eqn: main probability}}
\label{table: key parameters1}
\end{center}}

From \cref{fig: synthetic fix r}, we observe that for all three types of data matrices, the condition number of the sketched sub-matrices grows, statistically, as the sampling size $m$ increases. The concentration variance also expands as $m$ gets bigger. Moreover, sampling the standard Gaussian data produces the best conditioning in comparison to the other two data types.

The test suggests the uniform $[0,1]$ gives the least satisfying condition number, and we now proceed to examine the reason. To do so, we calculate the main parameters in \eqref{eqn: main probability}. See the list in \cref{table: key parameters1}. While the ratio $\frac{L}{\ell}$ for the three cases are all similar, though the distortion term $\tau$ varies greatly. In particular, according to~\eqref{eqn: distortion}, $\tau$ depends on $\frac{\|{\bf H}^N \|_2}{\Tr({\bf H}^N)},\frac{\|{\bf H}^N \|_F}{\Tr({\bf H}^N)}$, and it is clear that the uniform $[0,1]$ distribution produces much larger values for these factors than the Gaussian case.

\noindent\textbf{Fixed $m$ with varying $r$.} We then carry out another set of experiments, with the sample size $m$ fixed and we are free to adjust $r$ (conditioned that $r>m$). By \eqref{eqn: main probability}, the probability of the well-conditioned Hessian is expected to decrease as we increase the rank $r$. To demonstrate this, we display the frequency histogram for the rank of ${\bf H}^N_s \in \R^{m \times m}$ in \cref{fig: synthetic fix m}. As the full discrete Hessian ${\bf H}^N \in \R^{N \times N}$ \eqref{eqn:hessian_N} rank $r$ rises, we expect the sampled Hessian ${\bf H}^N_s$ \eqref{eqn: sub hess} rank to concentrate towards the size of $m,$ its largest possible rank. For this test, we only generate synthetic data using the Gaussian distribution. In \cref{fig: synthetic fix m}, we preset a threshold to compute the rank of the subsampled matrix ${\bf H}^N_s$, and the histogram demonstrates the distribution of the rank generated from $10,000$ runs using different $r$ values.

\noindent
\begin{minipage}{0.99\textwidth}
\centering
\captionsetup{type=figure}\addtocounter{figure}{-1}\subcaptionbox{Rank threshold $10^{-6}$. The computed rank is the number of eigenvalues that are above this threshold factored by the top eigenvalue $\lambda_1$.} 
[\linewidth]{\includegraphics[width=16cm]{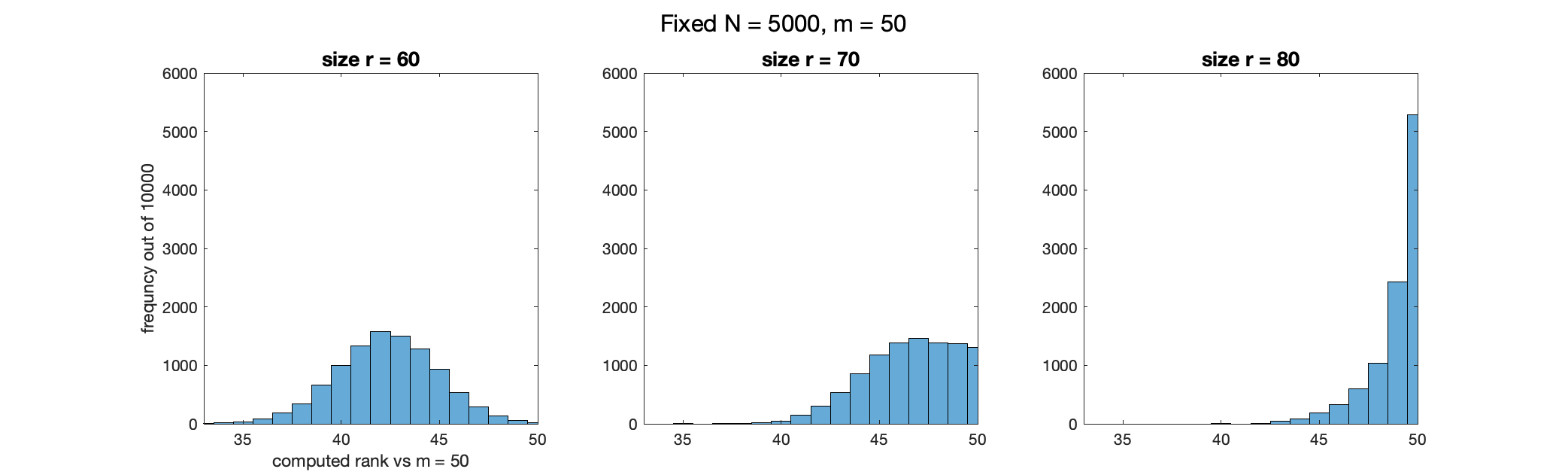}}\\
\subcaptionbox{Rank threshold $10^{-2}.$ The rank counting way is the same as in (a).}
[\linewidth]{\includegraphics[width=16cm]{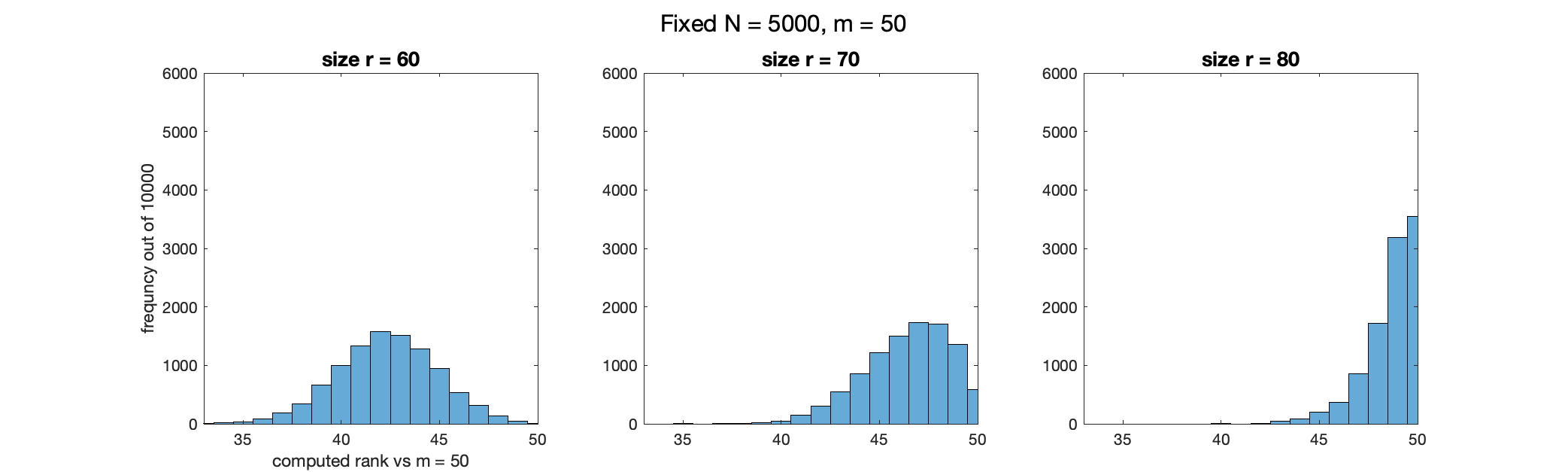}}    
\captionof{figure}{The rank distribution of Hessian submatrix ${\bf H}^N_s$. }
\label{fig: synthetic fix m}

\vspace{24pt}
\small
\hspace*{-0.5cm}
\begin{tabular}{ccccc} 
\toprule
\vspace{1mm}
 & $r = 60 $  & $r = 70$  & $r = 80$ & \\
 \vspace{1mm}
 $\frac{L}{\ell}$ &6.0995 &5.6329 &5.2802 & $\downarrow$\\
$\frac{\|{\bf H}^N \|_2}{\Tr({\bf H}^N)}$ &0.0597
& 0.0555&0.0523& $\downarrow$
\\
$\frac{\|{\bf H}^N \|_F}{\Tr({\bf H}^N)}$ &0.1707
&0.1636&0.1582 & $\downarrow$\\
\toprule
\end{tabular}
\captionof{table}{\small Key parameters in \eqref{eqn: main probability} (averaged over 10000 simulations).}
\label{table: key parameters2}
\end{minipage}
\vspace{24pt}

\cref{fig: synthetic fix m} shows that when $m$ is fixed at $50$, as the original rank $r$ increases, the distribution of the rank of the subsampled Hessian matrix also concentrates to $m$, the largest possible rank. This validates the statement in \cref{thm: well-conditioned hess} that as $r$ increases, one has higher probability of obtaining well-conditioned sub-matrices. The two rows in \cref{fig: synthetic fix m} showcases the cut-offs with different threshold. The top panel sets a more lenient rank condition and hence provides greater rank probability for the Hessian. Other key parameters in \eqref{eqn: main probability} are also provided in \cref{table: key parameters2}. Their values drop according to the increasing $r,$ which may lower the condition number threshold $\frac{L+\tau}{\ell-\tau}$ \eqref{eqn: main probability}-\eqref{eqn: distortion} as well.

\subsection{Elliptic inverse problem}
We now consider an inverse problem associated with elliptic type equations. This is similar to the classical EIT (electrical impedance tomography) problem from medical imaging but enjoys a better stability. This is because we use interior sensors to avoid the ill-conditioning issue commonly seen in the EIT problem \cite{AV05}. Consider the elliptic equation 
 \[
 \left\{\begin{array}{rl}
\nabla_x\cdot(\sigma\nabla_x u)&=S,\quad x\in\Omega\\
u&=0,\quad x\in\partial\Omega.\end{array}\right.\]
 The task is to provide sources to the elliptic equation and take measurement of the solution inside the domain, to reconstruct the unknown media function $\sigma^*\in L^2(\Omega)~(\Omega \subset \R^2)$, the elliptic coefficient, using finite measurements. We choose the ground-truth $\sigma^*$ to be the Shepp-Logan phantom \cite{SL74}. We define the loss function \eqref{eqn:opt_dis} to be the mismatch between the PDE-generated data and the true data. Suppose we fix the measurement index $j = (j_1, j_2)$. In particular, given the source $S_{j_1}$ to the PDE, we compute the PDE solution, and test the solution by the testing function $S_{j_2}$. In this setting, the following $u_{j_1}$ and $h_{j_2}$ solve the forward and adjoint models respectively, and are used for the construction of the loss function derivative:
\begin{equation}
\label{eqn: EIT}
\left\{\begin{array}{rl}
\nabla_x\cdot(\sigma\nabla_x u_{j_1})&=S_{j_1},\quad x\in\Omega\\
u_{j_1}&=0,\quad x\in\partial\Omega,
\end{array}\right.\quad \left\{
\begin{array}{rl}
\nabla_x \cdot(\sigma\nabla_x h_{j_2})&=S_{j_2},\quad x\in\Omega\\
h_{j_2}&=0,\quad x\in\partial\Omega.
\end{array}\right.
    \end{equation}
    The forward map $\mathcal{M}_j: L^2(\Omega) \to L^2(\Omega)$ is given by $\mathcal{M}_j[\sigma]\equiv d_{j_1,j_2}.$

\begin{figure}[!htp]
    \centering    \includegraphics[width=15cm]{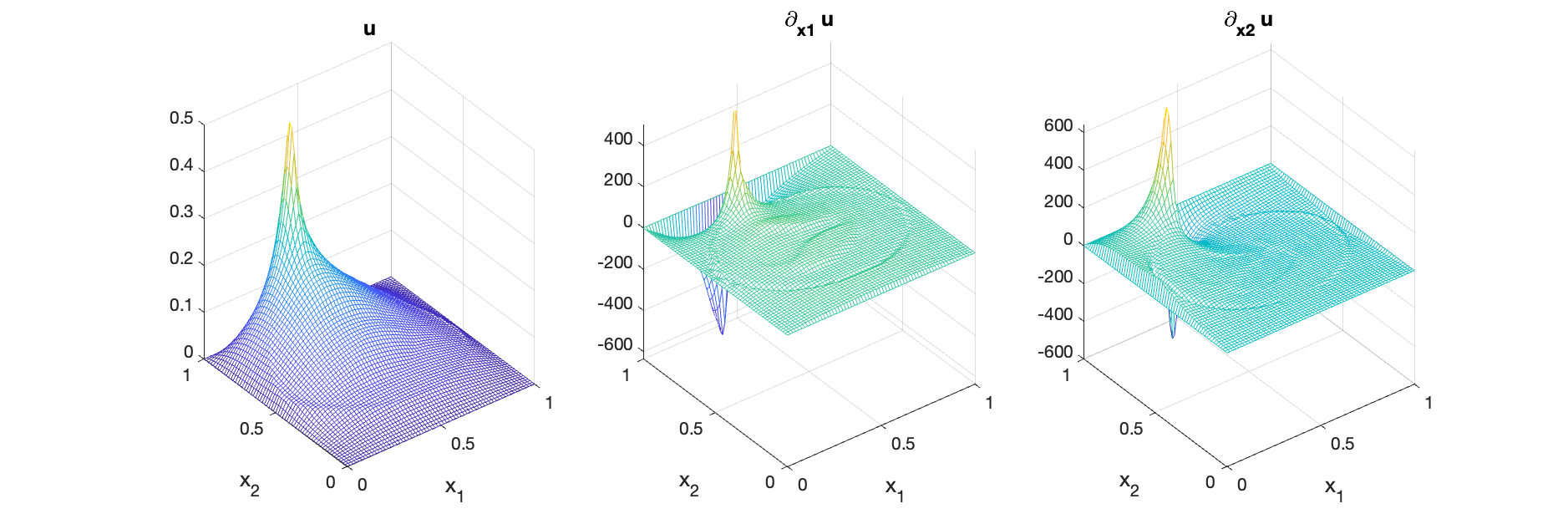}
    \caption{One solution to the forward model \eqref{eqn: EIT} at certain source location.}
\end{figure}

For the implementation, the window functions are chosen to be Gaussian functions as mollified versions of Dirac deltas:
\[
S_j(x) = e^{-0.1 \left(x- y_j\right)^2 },\, \text{for}~j = 1, \dots N.
\]
Moreover, the orthonormal basis $\{\phi_i\}_{i=1}^r$ is used to estimate the media function $\sigma.$ With all these preparations, we recall the notation defined in~\eqref{eqn: hess N} for the following Hessian matrix: 
\begin{equation}
\label{eqn: Phi EIT}
\begin{array}{ll}
{\bf H}^N = \pmb{\Phi}^N \pmb{\Phi}^{N\top}\,,\quad \text{where}~\pmb{\Phi}_{i,j} &\displaystyle = \langle \phi_i, \frac{\delta \mathcal{M}_j}{\delta \sigma} \rangle,\\& \displaystyle = \int_\Omega \partial_x \phi_i \, \nabla_x u_{j_1} \cdot \nabla_x h_{j_2} \dd x.
\end{array}
\end{equation}
Noting that the Fr\'{e}chet derivative takes the form of
$\nabla_x u \cdot\nabla_x h$, we then have 
\[
\pmb{\Phi}_{x,j} = \nabla_x u_{j_1}(x) \cdot \nabla_x h_{j_2}(x).\]

Numerically we set $\Omega =[0,1]^2 \subset \R^2$, and the mesh size is set to be $\Delta x= \frac{1}{64}$ in both dimensions. So the total number of grid points is $N = (64+1)^2 = 4225$. As for the measurements, we consider placing sources and detectors in two types of subdomains $\mathcal{D}_1,\mathcal{D}_2 \subset \Omega$, see \cref{fig: mesurement configuration}: 
\begin{enumerate}
    \item $ \mathcal{D}_1 = [\frac{1}{8},\frac{7}{8}]^2$. In this domain we have $r_1 = 384$ grid points for sample sources and detectors.
    \item $ \mathcal{D}_2 = [\frac{1}{32},\frac{31}{32}]^2 \setminus [\frac{3}{16},\frac{13}{16}]^2$. In this domain, we have $r_2 = 334$ grid points for sample sources and detectors.
    \end{enumerate}
To select data, we randomly pick $\frac{1}{6}$ portion of the nodes in these subdomains $\mathcal{D}$ to serve as sources, and we place detectors within $\pm 5$ spatial units from the source points. The generated discrete data matrices $\pmb{\Phi}^N, {\bf H}^N$ are presented in \cref{fig: data}. 

\begin{figure}[!htp]
    \centering
\includegraphics[width=15cm]{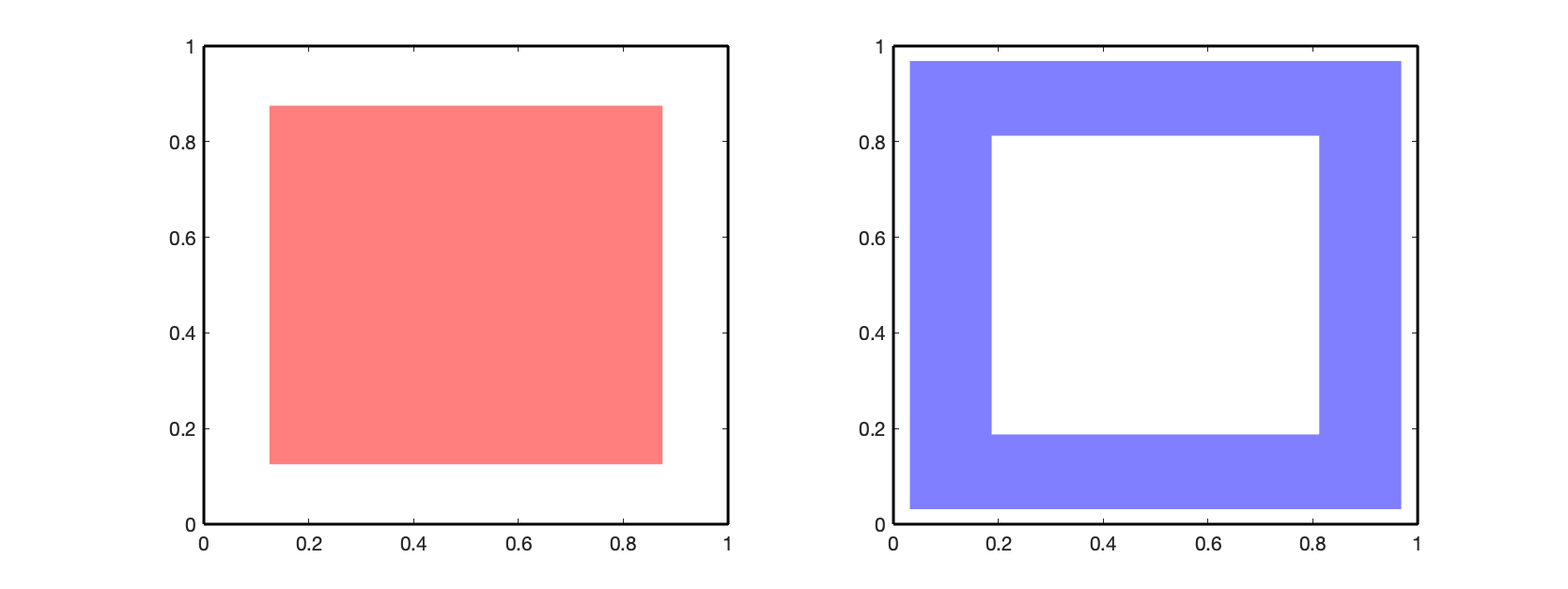}
\caption{Measurement layouts. The two layouts present sub-domains $\mathcal{D}_1$ and $\mathcal{D}_2.$}
\label{fig: mesurement configuration}
\end{figure}

\begin{figure}[!htp]
    \centering
\subfloat[Data $\pmb{\Phi}^N_{1}, {\bf H}^N_1$ with measurements from $\mathcal{D}_1$ ]{\includegraphics[width=9cm]{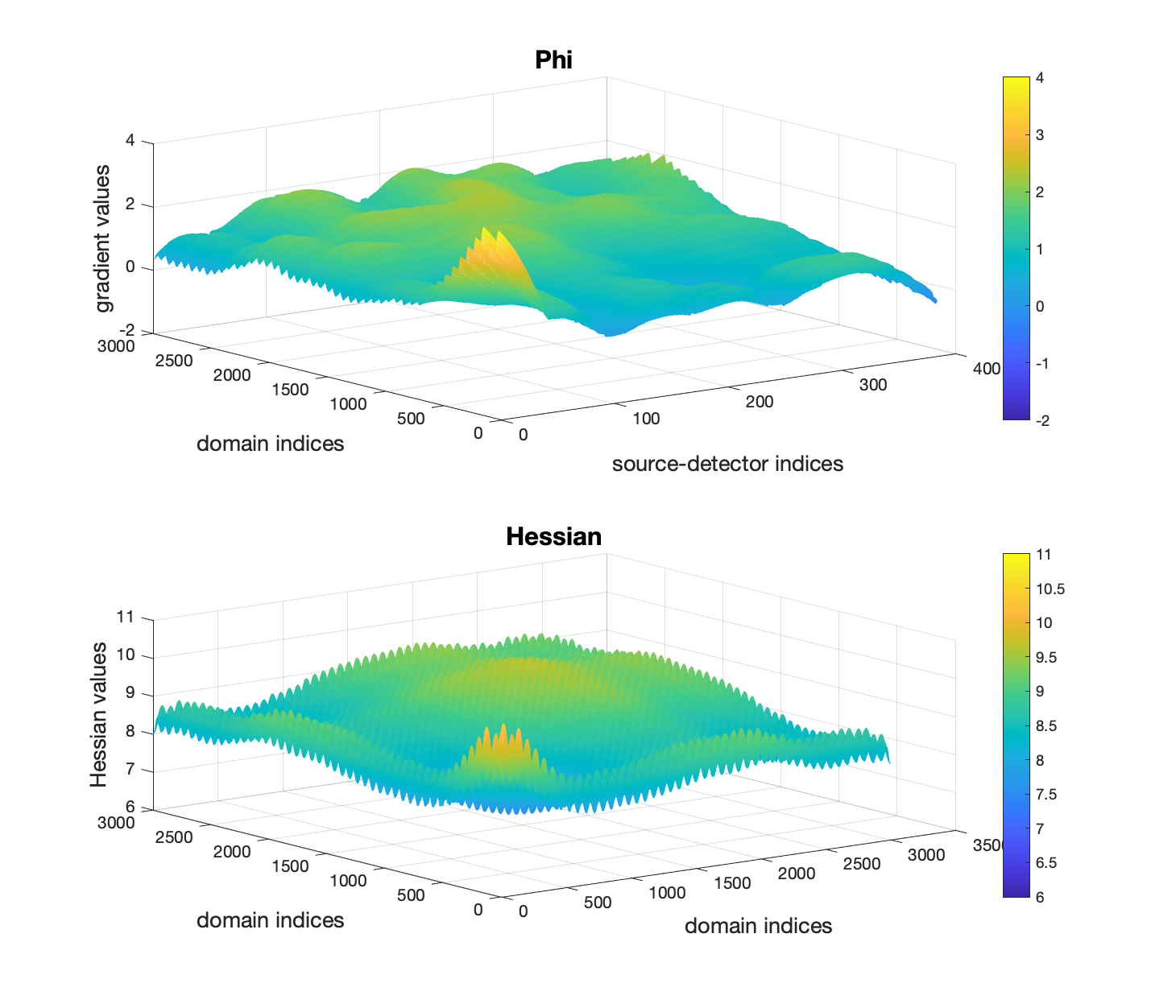}}
\subfloat[Data $\pmb{\Phi}^N_{2}, {\bf H}^N_2$ with measurements from $\mathcal{D}_2$]{\includegraphics[width=9cm]{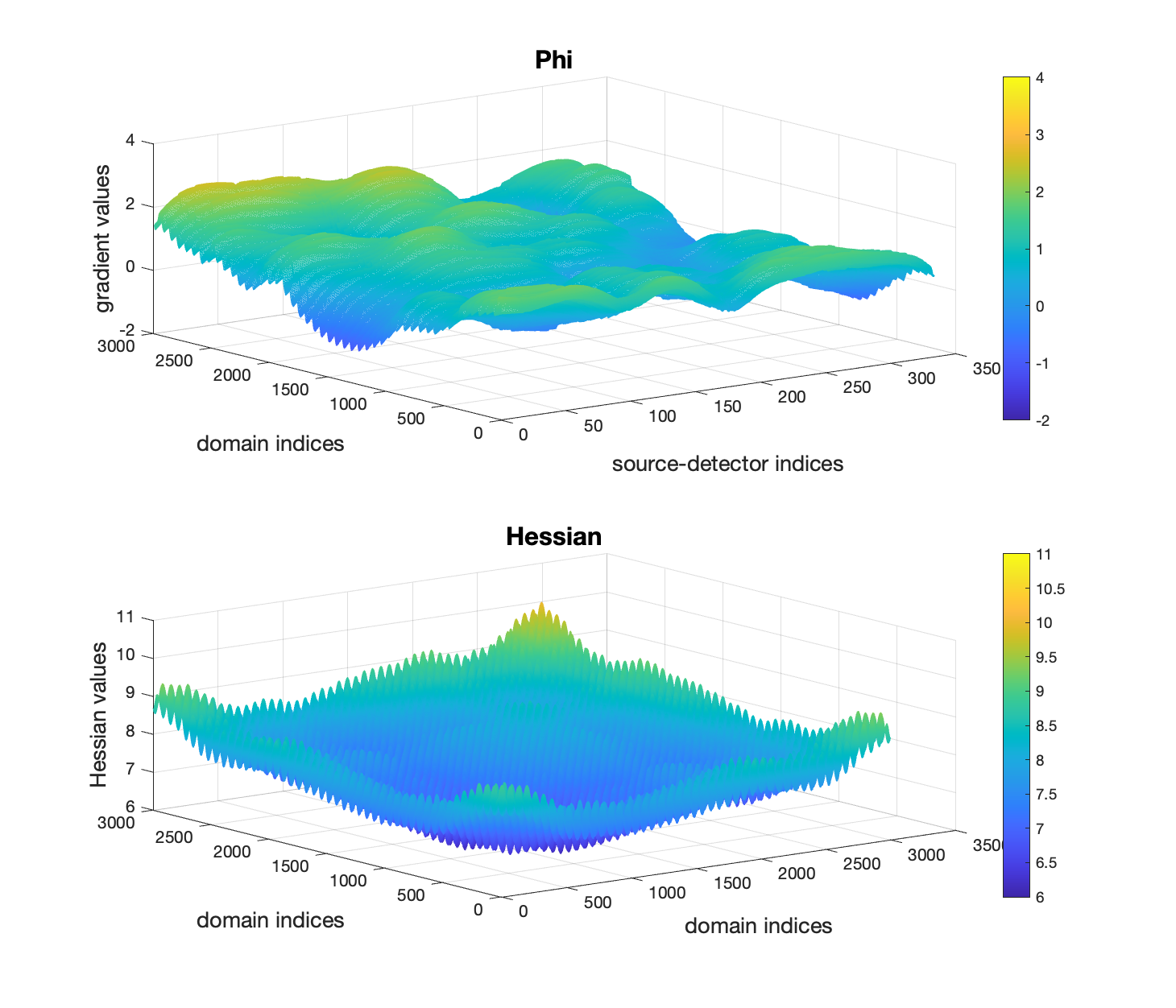}}
\caption{The entry magnitude of matrices $\pmb{\Phi}^N$ (top) and ${\bf H}^N$ (bottom). The left and right panels are results for $\mathcal{D}_1$ and $\mathcal{D}_2$ respectively. The magnitudes are showed in logarithm values, where the data matrices are first processed by Gaussian filters with std 20.}
\label{fig: data}
\end{figure}

We perform the random sketching on these two  established data matrices. In \cref{fig: eit condition1} and \cref{fig: eit condition2} we present the condition number of the subsampled matrices ${\bf H}_s^N \in \R^{m \times m}$ \eqref{eqn: sub hess} as a function of $m$. In each scenario, $10,000$ runs are performed and we document the $20\%-80\%$ quantile. As $m$ increases, the condition number clearly increases as well. The dotted line in the plots are calculated using $\frac{L}{\ell}$ \eqref{eqn: main probability} as a threshold. The two tables below the plots count the probability of the condition values exceeding this threshold from all simulations. The failure probability gradually increases as $m$ increases.

\begin{minipage}{0.49\textwidth}
\centering
\includegraphics[width=8cm]{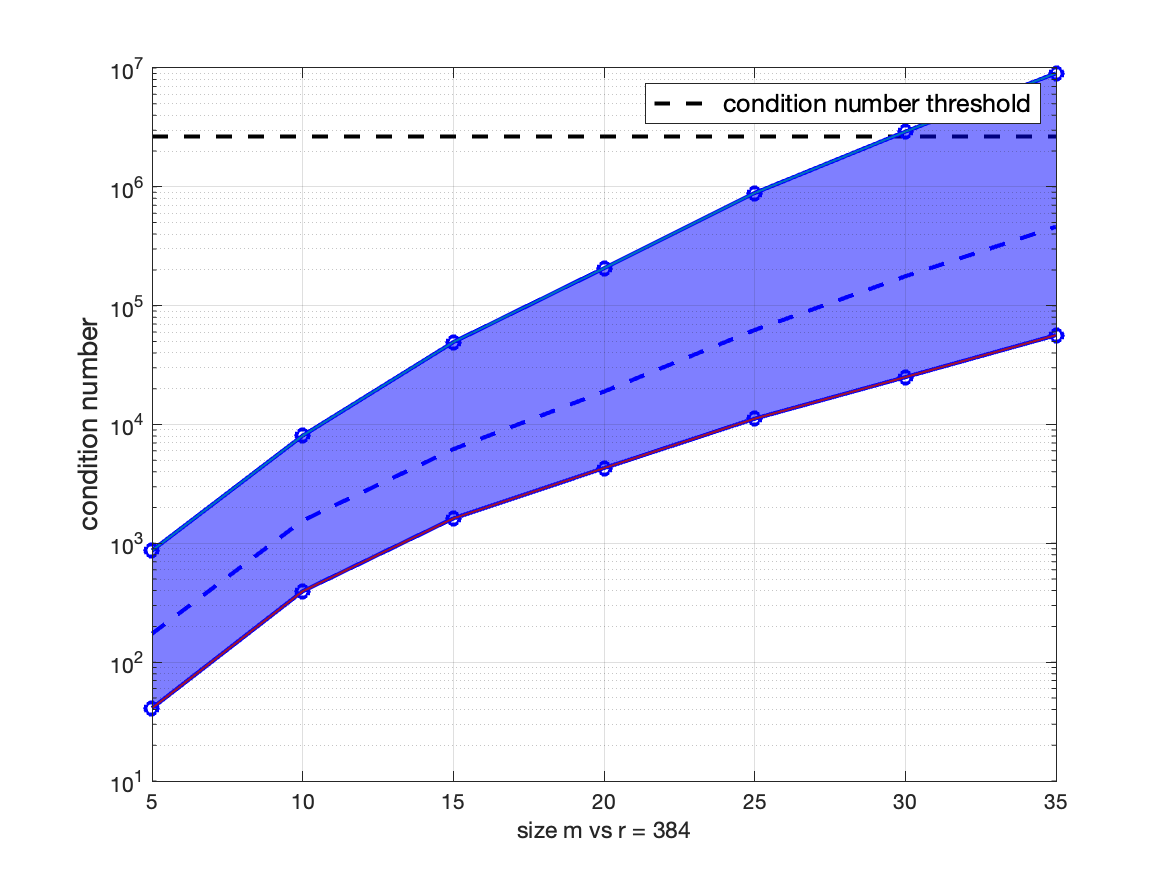}
\captionof{figure}{\small{Conditioning concentration of subsampled ${\bf H}^N_1$}}
\label{fig: eit condition1}
\vspace{24pt}
\scalebox{0.6}{
\centering
\begin{tabular}{cccccccc} 
\toprule
size $m$ (vs $r = 384$) & 5 & 10 & 15 & 20 & 25 & 30 & 35 \\
failure probability (\%)  &0.13 & 0.95&3.28&7.05&13.39&20.69&29.76\\ 
\toprule
\end{tabular}}
\captionof{table}{\small{LHS probability in \eqref{eqn: main probability} of ${\bf H}^N_1$}}
\label{table: failure 1}
\end{minipage}
\begin{minipage}{0.02\textwidth}
\textcolor{white}{bbbbbbbb}
\end{minipage}
\begin{minipage}{0.49\textwidth}
\centering
\includegraphics[width=8cm]{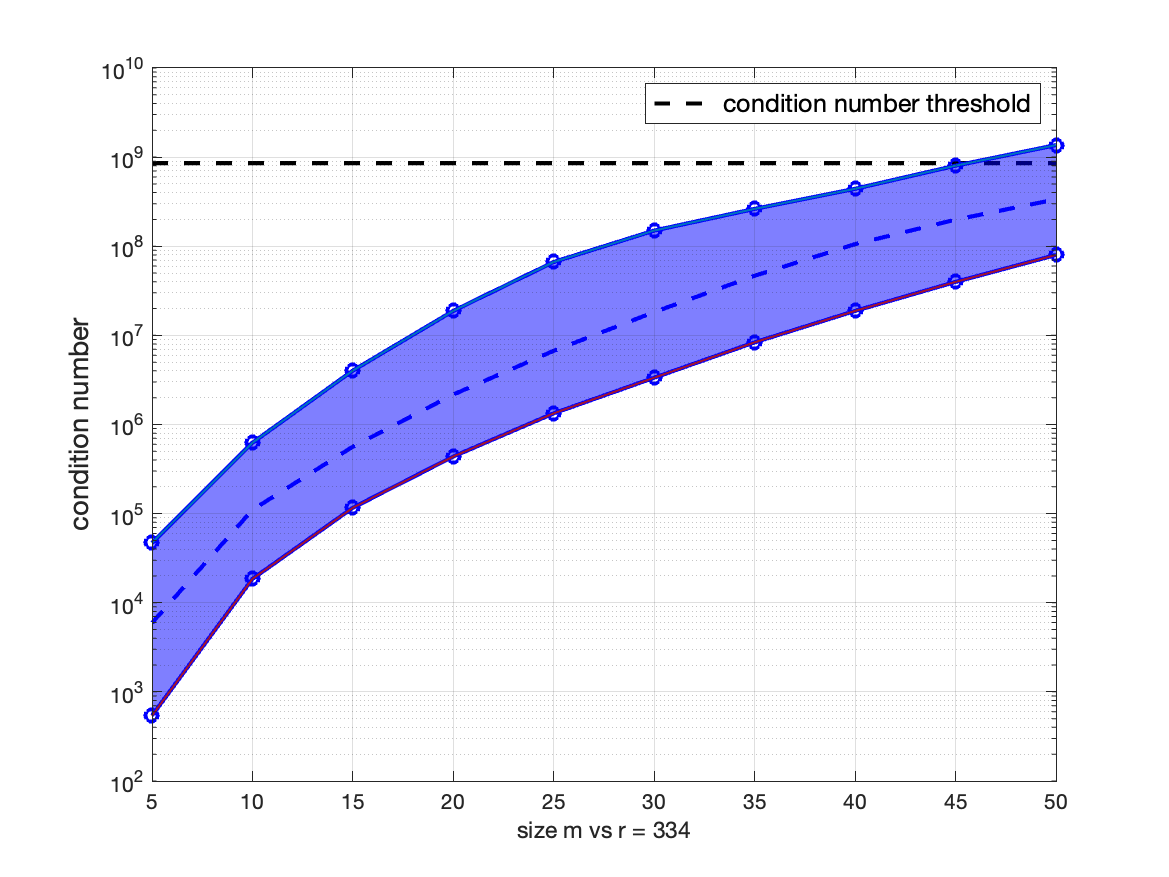}
\captionof{figure}{\small{Conditioning concentration of subsampled ${\bf H}^N_2$}}
\label{fig: eit condition2}
\vspace{24pt}
\scalebox{0.6}{
\centering
\begin{tabular}{cccccccccc} 
\toprule
size $m$ (vs $r = 334$) &10&15& 20 & 25 & 30 & 35 & 40 & 45 & 50 \\
failure probability (\%)  & 0.13 &0.2 &0.47 & 1.5&3.27&6.00&11.13&18.83&27.93\\ 
\toprule
\end{tabular}}
\captionof{table}{\small{LHS probability in \eqref{eqn: main probability} of ${\bf H}^N_2$}}
\label{table: failure 2}
\end{minipage}
\vspace{12pt}

These findings validate the main \cref{thm: well-conditioned hess}.
\section{Conclusions}
In a numerical inverse problem, we are usually only presented with a finite amount of data. It is unrealistic to expect that the limited information from the data can be used to fully reconstruct high dimensional parameters in large-scale physical systems, and thus one has to narrow down the goal to inferring a few subsampled parameters. In contrast to the unique reconstruction theory used to infer parameters in the infinite-dimensional setting, this paper studies the unique reconstruction in the finite dimensional setting, and spells out the number of reconstructable parameters ($m$) when the data given is fixed ($r$). In particular, we deploy random sketching techniques to examine the conditioning of the Hessian matrix that is $m\times m$ with rank $r$. When the ``success" is quantified by the Hessian's condition number, we establish the success probability. The theory is validated for cases of both synthetic data and PDE simulation data from an elliptic inverse problem. The numerical results confirm the theoretical findings: As long as the number of samples $m$ is within a constant factor of the full Hessian rank $r$, the sampled Hessian is well-conditioned with high probability. 

There are many related research directions that we would like to explore in the future. One such direction is to integrate the greedy approach, and to use the selected rows to eliminate the rows of similar information getting repeated. Furthermore, in our entire derivation, we have assumed the size of the Hessian is finite. We have noticed that the main \cref{thm: well-conditioned hess} has no explicit dependence on the ambient dimension $N$ but only on the rank and sampling size, $r$ and $m$. This is a strong indication that the result can be pushed to justify the scenario where $N=\infty$, in which case the Hessian is a bilinear functional, mapping $L^2 \times L^2$ to a real number. This would allow us to fully come back to the continuous PDE setting. Our analysis strongly depends on existing theory developed in~\cite{tropp2012user}. We believe with refined probabilistic arguments, the error bound can be tightened further as well.

\section*{Acknowledgements}
R.J. and Q.L. acknowledge the support from NSF-DMS-2308440. The research of S.S. is partially supported by NSF-DMS–2324368 and ONR grant N00014-21-1-2119. R.J. is partially supported by NSF-DMS-2023239. 
\appendix
\section{Positive-definiteness of Hessian} 
\label{hess positivity}
In this section, we give the proof details to show that the Hessian operator $\mathcal{H} \loss[\cdot]$ is strictly positive-definite in the neighbourhood of $\sigma_*$, i.e. \eqref{eqn:hessian_infinite}. Consequently $\sigma_*$ is a local unique optimum to the loss function \eqref{eqn:opt_cont}. 
\begin{proposition}
Assuming that 1. the Hessian operator $\mathcal{H}\loss [\cdot]$ is continuous in the neighborhood of $\sigma_*$; 2. the function set $\{\frac{\delta \mathcal{M}[\sigma_*](x)}{\delta \sigma}: \Omega \to \R \vert x \in \mathcal{D}\}$ spans the entire $L^2(\Omega).$ Then $\mathcal{H}\loss[\cdot]$ is strictly positive-definite at $\sigma_*$ and in the neighbourhood of $\sigma_*.$
\end{proposition}
\begin{proof}

We first prove that the Hessian operator $\mathcal{H} \loss[\cdot]$ is positive-definite at $\sigma_*$. For any $\widetilde{\sigma} \in L^2(\Omega),$ the Hessian of the optimization in \eqref{eqn:opt_cont} is derived as 
\begin{equation}
\label{eqn: hess loss}
\begin{array}{ll}
\mathcal{H} \loss[\sigma_*] (\widetilde{\sigma},\widetilde{\sigma}) & \displaystyle =  2 \int_{\mathcal{D}} \left\langle \frac{\delta \mathcal{M}[\sigma_*](x)}{\delta \sigma}, \widetilde{\sigma} \right\rangle_{\Omega}^2 \dd x +2 \mathcal{H} \mathcal{M}[\sigma_*] (\widetilde{\sigma},\widetilde{\sigma}) (\mathcal{M}[\sigma_*] - \mathcal{M}[\sigma_*])\\
& \displaystyle = 2 \int_{\mathcal{D}} \left\langle \frac{\delta \mathcal{M}[\sigma_*](x)}{\delta \sigma}, \widetilde{\sigma} \right\rangle_{\Omega}^2 \dd x.     
\end{array}
\end{equation}

Suppose there exists a non-zero function $\widetilde{\sigma}_0 \in L^2(\Omega)$ that makes the Hessian value \eqref{eqn: hess loss} be 0. Then it would also be the case that 
\[
\left\langle \frac{\delta \mathcal{M}[\sigma_*](x)}{\delta \sigma}, \widetilde{\sigma}_0 \right\rangle_{\Omega}=0, ~\forall ~x \in \mathcal{D}.\]
Due to Assumption 2 that \[\left\{\frac{\delta \mathcal{M}[\sigma_*](x)}{\delta \sigma}: \Omega \to \R ~\vert~x \in \mathcal{D}\right\} = L^2(\Omega),\] it implies that $\widetilde{\sigma}_0$ must be the zero function. This contradicts the above presumption of the Hessian value being 0 for a non-zero function. Therefore, the Hessian value $\mathcal{H}\loss[\sigma_*]$ \eqref{eqn: hess loss} must be strictly above 0 for all non-trivial $\widetilde{\sigma} \in L^2(\Omega).$

Moreover, due to Assumption 1 that $\mathcal{H} \Loss[\cdot]$ is continuous around $\sigma_*,$ we conclude that the Hessian operator is strictly positive-definite in the neighborhood of $\sigma_*.$
\end{proof}

\bibliographystyle{plain}
\bibliography{Reference}

\end{document}